\documentclass[11pt]{amsart}
\usepackage[latin9]{inputenc}
\usepackage{geometry}
\usepackage{float}
\usepackage{dvipost}
\usepackage{units}
\usepackage{mathtools}
\usepackage{amsthm}
\usepackage{amstext}
\usepackage{amssymb}
\usepackage{graphicx}
\usepackage{setspace}
\PassOptionsToPackage{normalem}{ulem}
\usepackage{ulem}

\makeatletter

\dvipostlayout
\dvipost{osstart color push Red}
\dvipost{osend color pop}
\dvipost{cbstart color push Blue}
\dvipost{cbend color pop}
\DeclareRobustCommand{\lyxadded}[3]{\changestart#3\changeend}
\DeclareRobustCommand{\lyxdeleted}[3]{%
\changestart\overstrikeon#3\overstrikeoff\changeend}

\numberwithin{equation}{section}
\numberwithin{figure}{section}
\usepackage{enumitem}		
\theoremstyle{plain}
\newtheorem{thm}{\protect\theoremname}[section]
  \theoremstyle{plain}
  \newtheorem{lem}[thm]{\protect\lemmaname}
  \theoremstyle{definition}
  \newtheorem{defn}[thm]{\protect\definitionname}
  \theoremstyle{remark}
  \newtheorem*{rem*}{\protect\remarkname}
  \theoremstyle{plain}
  \newtheorem{cor}[thm]{\protect\corollaryname}
  \theoremstyle{plain}
  \newtheorem{prop}[thm]{\protect\propositionname}

\usepackage{amsthm}\usepackage{bm}
\usepackage{verbatim}
\usepackage{mathrsfs}

\makeatother

  \providecommand{\corollaryname}{Corollary}
  \providecommand{\definitionname}{Definition}
  \providecommand{\lemmaname}{Lemma}
  \providecommand{\propositionname}{Proposition}
  \providecommand{\remarkname}{Remark}
\providecommand{\theoremname}{Theorem}

\begin{document}
\global\long\def\B{\mathcal{B}}

\global\long\def\R{\mathbb{R}}

\global\long\def\Q{\mathbb{Q}}

\global\long\def\Z{\mathbb{Z}}

\global\long\def\C{\mathbb{C}}

\global\long\def\N{\mathbb{N}}

\global\long\def\at{\mathbf{\big|}}

\global\long\def\id{\mathbf{1}}

\global\long\def\ui{\mathbf{\textrm{i}}}

\global\long\def\ue{\mathbf{\textrm{e}}}

\global\long\def\ud{\mathbf{\textrm{d}}}

\global\long\def\nup{\overline{N}}

\global\long\def\ndown{\underline{N}}

\global\long\def\N{\mathbb{N}}

\global\long\def\Z{\mathbb{Z}}

\global\long\def\R{\mathbb{R}}

\global\long\def\Q{\mathcal{Q}}

\global\long\def\O{\mathcal{O}}

\global\long\def\E{\mathcal{E}}

\global\long\def\P{\mathcal{P}}

\global\long\def\ef{\varphi}

\global\long\def\trngle{\mbox{\ensuremath{\mathcal{D}}}}

\global\long\def\hftrngle{\frac{1}{2}\mbox{\ensuremath{\mathcal{D}}}}

\global\long\def\dom{\Omega}

\global\long\def\sdom{\mathcal{S}}

\global\long\def\recdom#1{\mathcal{R}^{\left(#1\right)}}

\global\long\def\mvec{\vec{m}}

\global\long\def\xvec{\vec{x}}

\global\long\def\nbox{\mbox{\ensuremath{\mathcal{B}}}^{\left(n\right)}}

\global\long\def\hfbox{\frac{1}{2}\mbox{\ensuremath{\mathcal{B}}}^{\left(n\right)}}

\global\long\def\ohm{\gamma_{n}}

\global\long\def\ohmdmn{\gamma}

\global\long\def\sp#1{\textrm{Span}_{\mathbb{Q}}\left\{  #1\right\}  }

\global\long\def\lams{\Lambda^{\left(n\right)}}

\global\long\def\bas{\mathcal{G}{}^{\left(n\right)}}

\global\long\def\qbox#1{\mathcal{B_{\mathcal{Q}}}\left(#1\right)}

\global\long\def\qtrngle#1{\mathcal{T_{\mathcal{Q}}}\left(#1\right)}

\global\long\def\spec#1{\textrm{\ensuremath{\sigma}}\left(#1\right)}

\global\long\def\ospec#1{\textrm{\ensuremath{\sigma}}_{\textrm{odd}}\left(#1\right)}

\global\long\def\espec#1{\textrm{\ensuremath{\sigma}}_{\textrm{even}}\left(#1\right)}

\global\long\def\oval{\lambda^{\left(0\right)}}

\global\long\def\pr#1{\textrm{\ensuremath{\mathfrak{p}}}\left(#1\right)}

\global\long\def\fold{\mathbf{F}}

\global\long\def\unfold{\mathbf{U}}

\global\long\def\bdry{\underrightarrow{\partial}\Q}

\global\long\def\line{\mathrm{L}}

\global\long\def\mult{\mathrm{d}}

\global\long\def\vol#1{\textrm{vol}\left(#1\right)}

\global\long\def\mcl#1{\mathcal{#1}}

\global\long\def\nodaldef#1{\mathcal{\delta}\left(#1\right)}

\title[Courant-sharpness of 2-Rep-Tiles]{Courant-sharp Eigenvalues of Neumann 2-Rep-tiles}

\author{Ram Band$^{1}$, Michael Bersudsky$^{1}$ and David Fajman$^{2}$}

\address{$^{1}${\small{}Department of Mathematics, Technion--Israel Institute
of Technology, Haifa 32000, Israel}}

\address{$^{2}${\small{}Faculty of Physics, University of Vienna, Boltzmanngasse
5, 1090 Vienna, Austria}}


\begin{abstract}
We find the Courant-sharp Neumann eigenvalues of the Laplacian on
some 2-rep-tile domains. In $\R^{2}$ the domains we consider are
the isosceles right triangle and the rectangle with edge ratio $\sqrt{2}$
(also known as the A4 paper). In $\R^{n}$ the domains are boxes which
generalize the mentioned planar rectangle. The symmetries of those
domains reveal a special structure of their eigenfunctions, which
we call folding\textbackslash{}unfolding. This structure affects the
nodal set of the eigenfunctions, which in turn allows to derive necessary
conditions for Courant-sharpness. In addition, the eigenvalues of
these domains are arranged as a lattice which allows for a comparison
between the nodal count and the spectral position. The Courant-sharpness
of most eigenvalues is ruled out using those methods. In addition,
this analysis allows to estimate the nodal deficiency - the difference
between the spectral position and the nodal count.
\end{abstract}

\maketitle

\section{Introduction\label{sec:Introduction}}

\noindent It is nearly a century since Courant proved his famous nodal
result, stating that the $n^{\textrm{th }}$ Laplacian eigenfunction
cannot have more than $n$ nodal domains \cite{Cou_ngwgmp23}. Eigenfunctions
which achieve this upper bound are called \emph{Courant-sharp} and
it was Pleijel who showed that there are only finitely many of them
\cite{Pleijel_cpam56}. Much more recently Polterovich proved a similar
result for domains with Neumann boundary conditions \cite{Polterovich_ams09}.
In his paper, Pleijel also pointed out the Courant-sharp eigenfunctions
of the square with Dirichlet boundary conditions. What started in
this early work of Pleijel, is recently revived as a systematic search
of Courant-sharp eigenfunctions of various domains. Part of this analysis
owes to the broad interest in the general subject of nodal domains,
but this line of research in particular stems from the latest works
on nodal partitions. The search for Courant-sharpness was pioneered
by Helffer, Hoffmann-Ostenhof and Terracini, who found that minimization
of some energy functional over a set of domain partitions is connected
to nodal patterns of eigenfunctions \cite{HelHofTer_aihp09}. In particular,
they have shown that if the minimum of this functional over domain
partitions of $k$ subdomains is equal to the $k$'th eigenvalue,
then the $k$'th eigenvalue is Courant-sharp. In addition, the nodal
partition of the corresponding Courant-sharp eigenfunction is a minimizing
partition.\footnote{It is worthwhile to mention that a similar variational approach was
recently developed by Berkolaiko, Kuchment and Smilansky. Their results
also characterize the nodal sets of non-Courant-sharp eigenfunctions
\cite{BerKucSmi_gafa12}.}. This led to a particularized search for Courant-sharp eigenfunctions
of various domains over just the last couple of years. Among the domains
that were treated are the square, the disc, the annulus, irrational
rectangles, the torus and some triangles, where the analysis in those
cases is specialized to the considered boundary conditions (either
Dirichlet or Neumann). Most of those investigations are done by Helffer
collaborating with Hoffmann-Ostenhof and Terracini \cite{HelHof_chapter13,HelHof_jst14,HelHofTer_ims10},
with B{\'e}rard \cite{BerHel_ang15,BerHel_lmp16} and with Sundqvist
\cite{HelSun_ams16,HelSun_mmj15}. Additional results for various
tori are proved by L\'ena \cite{Lena_ams16,Lena_crm15}. For further
details and references we refer the reader to the recent reviews by
Bonnaillie-No\"el and Helffer \cite{BonHel_arXv15} and by Laugesen
and Siudeja \cite{LauSiu_book16}. While these reviews came out and
afterwards, three additional results, which for the first time concern
high-dimensional domains, were proven. Helffer and Kiwan determined
the Courant-sharp eigenfunctions of the cube \cite{HelKiw_arXv15},
L\'ena found them for the three-dimensional square torus \cite{Lena_ams16}
and Helffer with Sundqvist solved the problem for Euclidean balls
in any dimension \cite{HelSun_ams16}. Finally, in another direction,
Helffer and B{\'e}rard and also van den Berg and Gittins provided
bounds on the largest Courant-sharp Dirichlet eigenvalue and on the
total number of them for a general domain \cite{BerHel_arXiv15,VanGit_arXiv16}.

In the present work we determine the Courant-sharp eigenfunctions
of certain 2-rep-tile domains with Neumann boundary conditions. A
domain is said to be \emph{rep-tile} (replicating figure, a name
coined by Golomb \textbf{\cite{Golomb_tmg64}}) if it can be decomposed
into $k$ isometric domains, each of which is similar to the original
domain. According to the number of its subdomains ($k$), the domain
is called rep-$k$ or a $k$-rep-tile. The convex polygonal $2$-rep-tiles
in the plane are known to be the isosceles right triangle and parallelograms
with edge length ratio $\sqrt{2}$ \textbf{\cite{Martin_transformation_geometry_82,NgaSirVeeWan_ged00}}.
In this paper we find the Courant-sharp eigenfunctions of such Neumann
triangle (Theorem \ref{thm:CourantSharp_triangle}) and Neumann rectangle,
together with all the $2$-rep-tile high-dimensional boxes, which
generalize this rectangle (Theorem \ref{thm:CourantSharp_boxes}).
In addition to being $2$-rep-tiles, they all have the special property
that the cut which separates them into the mentioned two subdomains
serves also as a symmetry axis\footnote{Using involution symmetry for studying nodal counts may be found already
in early studies of Leydolod \cite{Leydold_Phd89,Leydold_top96}.} (or hyperplane for the boxes). Thus, for an eigenfunction which is
even with respect to the symmetry axis, its restriction to the subdomain,
when rescaled, yields again an eigenfunction of the same eigenvalue
problem. This allows us to identify a special structure, ordering
all of the eigenfunctions, which we call the \emph{folding} (\emph{unfolding})
structure. Using this classification we prove that all eigenfunctions
within a certain class vanish on the same subset.\footnote{This connects nicely to the recent works \cite{Agranovsky_arxv15,BouRud_inv11},
though those do not concern Courant-sharpness.}This property allows to rule out Courant-sharpness of eigenfunctions
without using the Faber-Krahn inequality \textbf{\cite{Faber_23,Krahn_ma25}}
or similar isoperimetric inequalities. Such isoperimetric inequalities
form the first step in ruling out Courant-sharpness in most of the
works mentioned above (with the exception of irrational rectangles,
the disk and Euclidean balls, where properties of minimal partitions
were used for that purpose). We present our result for the rectangle
in a general form (Theorem \ref{thm:CourantSharp_boxes}), valid for
all $n$-dimensional ($n\geq2)$ boxes which are 2-rep-tiles and symmetric
with respect to their hyperplane cut (see Figure \ref{fig:box and half box}).
It is interesting to note that the case of the rectangle goes beyond
the irrational rectangles which were explored so far, as its square
of edge ratio is rational (it equals two). Its spectrum is therefore
not simple as in the case of the irrational rectangles (treated in
\cite{HelHofTer_aihp09}). Yet, the folding structure mentioned above
allows to quickly rule out all of its multiple eigenvalues, and the
same goes for all high-dimensional boxes.

The outline of the paper is as follows. This section continues by
providing useful notations and exact statements of our main results.
We then present the so\lyxadded{mikein}{Wed Nov 02 13:00:36 2016}{-}\lyxdeleted{mikein}{Wed Nov 02 13:00:36 2016}{
}called folding structure for eigenfunctions of the isosceles right
triangle in Section \ref{sec:folding_structure_triangle}. In Section
\ref{sec:proof_of_triangle} we complete the investigation of the
triangle's eigenfunctions and prove Theorem \ref{thm:CourantSharp_triangle}.
In Section \ref{sec:proof_for_boxes} we present the folding structure
for the box eigenfunctions and prove Theorem \ref{thm:CourantSharp_boxes}.
In Appendix \ref{sec:appendix_multiplicity_boxes} we present some
identities concerning eigenvalue multiplicities for the boxes, connecting
those to the two-dimensional problem of the rectangle. Finally, in
Appendix \ref{sec:appendix_nodal_deficiency} we go beyond Courant-sharpness
by describing some results on the nodal deficiency and in Appendix
\ref{sec:appendix_dirichlet_problem} we point out how some of our
methods apply for the same domains, but with Dirichlet boundary conditions.

\subsection{Notations and Preliminaries}

We consider the Laplacian eigenvalue problem on a bounded domain $\Omega$
with Neumann boundary conditions, 
\begin{equation}
-\Delta\varphi=\lambda\varphi,\qquad\frac{\partial\varphi}{\partial\vec{n}}\Big\vert_{\partial\Omega}=0.\label{eq:eigen_prob}
\end{equation}
 We denote the corresponding spectrum by $\spec{\Omega}$ and note
that it can be described by an increasing sequence of eigenvalues
\[
0=\lambda_{1}<\lambda_{2}\leq\lambda_{3}\leq..\nearrow\infty.
\]
We define the following spectral counting functions: 
\begin{equation}
\nup\left(\lambda\right):=\left|\left\{ j\,\at~\lambda_{j}\leq\lambda\right\} \right|,\label{eq:upper_counting}
\end{equation}

\begin{equation}
\ndown\left(\lambda\right):=\left|\left\{ j\,\at~\lambda_{j}<\lambda\right\} \right|,\label{eq:lower_counting}
\end{equation}

\begin{equation}
N\left(\lambda\right):=\begin{cases}
\ndown\left(\lambda\right) & ~~\lambda\notin\sigma\left(\Omega\right)\\
\ndown\left(\lambda\right)+1 & ~~\lambda\in\sigma\left(\Omega\right)
\end{cases}\label{eq:counting}
\end{equation}
and denote the multiplicity of an eigenvalue by
\begin{equation}
\mult\left(\lambda\right):=\nup\left(\lambda\right)-\ndown\left(\lambda\right).\label{eq:multiplicity}
\end{equation}
For an eigenfunction $\varphi$ on $\Omega$ we denote by $\nu\left(\varphi\right)$
the number of connected components of $\Omega\backslash\varphi^{-1}\left(0\right)$,
also known as \emph{nodal domains}.

In terms of those definitions, the celebrated Courant bound reads
$\nu\left(\varphi\right)\leq N\left(\lambda\right)$, where $\varphi$
is an eigenfunction of the eigenvalue $\lambda$ \cite{Cou_ngwgmp23}.
\\
Let $\varphi$ be an eigenfunction on $\Omega$ with eigenvalue $\lambda$.
We say that $\varphi$ is a \emph{Courant-sharp eigenfunction} if
$\nu\left(\varphi\right)=N\left(\lambda\right)$. In this case we
also say that $\lambda$ is a \emph{Courant-sharp eigenvalue.}

\subsection{Main results}
\begin{thm}
\label{thm:CourantSharp_triangle} The Courant-sharp eigenvalues of
the Neumann Laplacian on the isosceles right triangle are $\lambda_{1},\lambda_{2},\lambda_{3},\lambda_{4},\lambda_{6}$.
\end{thm}
~
\begin{thm}
\label{thm:CourantSharp_boxes} Let $n\in\N$, $n\geq2$, and let
$\nbox$ be an $n$-dimensional box of measures $l_{1}\times l_{2}\times\ldots\times l_{n}$,
where the ratios of edge lengths are given by $\frac{l_{j}}{l_{j+1}}=2^{\nicefrac{1}{n}}$
$(1\leq j\leq n-1)$. The Courant-sharp eigenvalues of the Neumann
Laplacian on $\nbox$are $\lambda_{1},\lambda_{2},\lambda_{4},\lambda_{6}$
for $n=2$ and $\lambda_{1},\lambda_{2}$ for $n\geq3.$
\end{thm}

\section{Eigenfunction Folding Structure of the triangle \label{sec:folding_structure_triangle}}

\noindent We consider the following scaling for the isosceles right
triangle, 
\[
\mcl D=\{(x,y)\in[0,\pi]\times[0,\pi]\,\at\,y\leq x\}.
\]

\noindent For geometric convenience to be exploited later, $\mcl D$
denotes the closed domain, and the Laplacian is defined on its interior,
$\Omega=\trngle^{\circ}$.

\noindent Denote $\mathbb{N}_{0}:=\N\cup\left\{ 0\right\} $ and define
the set 
\begin{equation}
\Q:=\left\{ \left.\left(m,n\right)\in\mathbb{N}_{0}\times\mathbb{N}_{0}\right|m\geq n\right\} ,\label{eq:quantum_numbers}
\end{equation}
which we call the set of \emph{quantum numbers.} A complete orthogonal
basis of eigenfunctions is given by 
\begin{equation}
\varphi_{m,n}(x,y)=\cos(mx)\cos(ny)+\cos(my)\cos(nx)\ ;\ \left(m,n\right)\in\Q,\label{eq:eigenfunction}
\end{equation}
and the spectrum is given by 
\[
\spec{\trngle}=\left\{ \lambda_{m,n}=\left\Vert \left(m,n\right)\right\Vert ^{2}\,\at\,\left(m,n\right)\in\Q\right\} ,
\]
where
\[
\left\Vert \left(m,n\right)\right\Vert ^{2}=m^{2}+n^{2}.
\]
It is useful to define 
\[
\Q(\lambda):=\left\{ \left(m,n\right)\in\Q\,\at\,\left\Vert \left(m,n\right)\right\Vert ^{2}<\lambda\right\} 
\]
and observe that
\[
\ndown\left(\lambda\right)=\left|\Q(\lambda)\right|.
\]
The isosceles right triangle $\trngle$ is symmetric with respect
to the median to the hypotenuse, 
\[
\line=\left\{ \left.\left(x,y\right)\in\trngle\right|x+y=\pi\right\} 
\]
and the symmetry is expressed by
\[
R\left(x,y\right)=\left(\pi-y,\pi-x\right).
\]
We describe in the following a special feature of eigenfunctions on
the triangle, which is based on the symmetry above.
\begin{lem}
\label{lem:eigenfunction_symmetry}\emph{ }Let $\lambda\in\sigma\left(\mathcal{D}\right)$,
then its corresponding eigenfunctions are odd (even) with respect
to $\line$\emph{ if and only if }$\lambda$ is odd (even)\emph{.}\end{lem}
\begin{proof}
Let $\lambda_{m,n}\in\sigma\left(\trngle\right)$, we get
\begin{align*}
\varphi_{m,n}\left(R\left(x,y\right)\right) & =\varphi_{m,n}(\pi-y,\pi-x)\\
 & =\cos(m\pi-my)\cdot\cos(n\pi-nx)+\cos(m\pi-mx)\cdot\cos\left(n\pi-ny\right)\\
 & =(-1)^{m+n}\left[\cos(mx)\cos(ny)+\cos(my)\cos(nx)\right]=(-1)^{m+n}\varphi_{m,n}\left(x,y\right),
\end{align*}

so that $\varphi_{m,n}$ is odd if and only if $m\neq n\,\,(\textrm{mod~2})$
and even if and only if $m=n\,\,(\textrm{mod~2})$. As $\lambda_{m,n}=m^{2}+n^{2}$
we get that $\varphi_{m,n}$ is odd if and only if $\lambda_{m,n}$
is odd and even if and only if $\lambda_{m,n}$ is even. The lemma
now follows since the elements of $\left\{ \varphi_{m,n}\right\} _{m^{2}+n^{2}=\lambda}$
form a basis for the eigenspace.
\end{proof}
Lemma \ref{lem:eigenfunction_symmetry} motivates the following definition.
\begin{defn}
\label{def:odd_even_spectrum}We define the subsets of $\Q$ that
correspond to the odd and even eigenvalues

\begin{equation}
\begin{array}{c}
\O\coloneqq\left\{ (m,n)\in\Q~|~m\neq n\,(\textrm{mod~2})\right\} \\
\\
\E\coloneqq\left\{ (m,n)\in\Q~|~m=n\,(\textrm{mod~2})\right\} 
\end{array}\label{eq:odd_even_qn}
\end{equation}
and we denote the corresponding sets of eigenvalues by
\[
\ospec{\trngle}\coloneqq\left\{ \lambda_{m,n}\at\left(m,n\right)\in\O\right\} 
\]
\[
\espec{\trngle}\coloneqq\left\{ \lambda_{m,n}\at\left(m,n\right)\in\E\right\} ,
\]
where in those sets each eigenvalue appears as many times as its multiplicity. 
\end{defn}
Denote 
\[
\hftrngle:=\left\{ \left(x,y\right)\in\trngle\at\left(x+y,x-y\right)\in\left[0,\pi\right]^{2}\right\} ,
\]
and observe that $\line$ partitions $\mathcal{D}$ into the two isometric
triangles $\hftrngle$ and $\left(\trngle\backslash\hftrngle\right)\cup\line$,
each is a scaled version of $\mathcal{D}$ by a factor $\sqrt{2}$.
The following defines the transformations which describe the similarity
relation between $\trngle$ and $\hftrngle$.
\begin{defn}
\label{def:fold_unfold_coords} We define the \emph{coordinate folding
transformation} 

\begin{align}
\begin{array}{c}
F:\hftrngle\to\trngle\\
\\
F\left(x,y\right)\coloneqq\left(x+y,x-y\right)
\end{array}\label{eq:fold_coordinates}
\end{align}
and the \emph{coordinate unfolding transformation} as the inverse
of $F$ by 
\begin{equation}
\begin{array}{c}
U:\trngle\to\hftrngle\\
\\
U\left(u,v\right)=\left(\frac{u+v}{2},\frac{u-v}{2}\right).
\end{array}\label{eq:unfold_coords}
\end{equation}
\end{defn}
\begin{rem*}
The mappings $F$ and $U$ are indeed similarity transformations between
$\trngle$ and $\hftrngle$ as
\[
F(x,y)=\underbrace{\sqrt{2}}_{\textrm{scaling}}\underbrace{\left(\begin{array}{cc}
\cos\frac{\pi}{4} & -\sin\frac{\pi}{4}\\
\sin\frac{\pi}{4} & \cos\frac{\pi}{4}
\end{array}\right)\left(\begin{array}{cc}
0 & 1\\
1 & 0
\end{array}\right)}_{\textrm{isometry}}\left(\begin{array}{c}
x\\
y
\end{array}\right)
\]
and
\end{rem*}
\[
U(x,y)=\underbrace{\frac{1}{\sqrt{2}}}_{\textrm{scaling}}\underbrace{\left(\begin{array}{cc}
\cos\frac{\pi}{4} & -\sin\frac{\pi}{4}\\
\sin\frac{\pi}{4} & \cos\frac{\pi}{4}
\end{array}\right)\left(\begin{array}{cc}
0 & 1\\
1 & 0
\end{array}\right)}_{\textrm{isometry}}\left(\begin{array}{c}
x\\
y
\end{array}\right).
\]
The next definition introduces the notion of folding and unfolding
of an eigenfunction.
\begin{defn}
\label{def:fold_unfold_efunc} Let $\varphi$ be an eigenfunction
corresponding to $\lambda\in\sigma\left(\mathcal{D}\right)$, \end{defn}
\begin{enumerate}
\item Assume $\lambda\in\espec{\trngle}$, we define the \emph{folded function,
$\fold\varphi:\,\trngle\rightarrow\R$, as} 
\begin{equation}
\fold\varphi\left(x,y\right)=\varphi\circ U\left(x,y\right)\,,\,\left(x,y\right)\in\trngle.\label{eq:fold_fcn}
\end{equation}

\item We define the \emph{unfolded function, $\unfold\varphi:\,\trngle\rightarrow\R$,
as }
\begin{equation}
\unfold\varphi\left(x,y\right)=\begin{cases}
\varphi\circ F\left(x,y\right) & \left(x,y\right)\in\hftrngle\\
\left(\varphi\circ F\right)\circ R\left(x,y\right) & \left(x,y\right)\in\mathcal{D}\backslash\hftrngle
\end{cases}.\label{eq:unfold_fcn}
\end{equation}

\end{enumerate}
Note that only folding of an even eigenfunction gives a new function
whose normal derivative vanishes on $\partial\trngle$. Therefore
it follows that only folding of an even eigenfunction results with
another eigenfunction. Unfolding of any eigenfunction, always results
with another eigenfunction. Hence we consider the foldings for the
even eigenfunctions and the unfoldings for all eigenfunctions. We
also remark that the folded (unfolded) eigenfunction is of an eigenvalue
which is twice as small (large), since the coordinate folding (unfolding)
transformation is a similarity transofrmation with a scaling factor
of $\sqrt{2}$ ~($1/\sqrt{2})$. Those results are stated and proved
below.
\begin{lem}
\label{lem:fold_unfold_efunc}Let 
\[
\varphi=\underset{k^{2}+l^{2}=\lambda}{\sum}\alpha_{k,l}\varphi_{k,l}
\]
 be an eigenfunction corresponding to the eigenvalue $\lambda$ ,
then the following holds: 
\begin{enumerate}
\item If $\lambda\in\espec{\trngle}$, then the folded function $\fold\varphi$
is an eigenfunction corresponding to the eigenvalue $\frac{\lambda}{2}$
and is given by\emph{
\begin{equation}
\begin{array}{c}
\fold\varphi=\underset{k^{2}+l^{2}=\lambda}{\sum}\alpha_{k,l}\varphi_{_{F_{\Q}\left(k,l\right)}},\\
\\
\textrm{where\,\,\,\, }F_{\Q}\left(k,l\right)=\left(\frac{k+l}{2},\frac{k-l}{2}\right).
\end{array}\label{eq:formula folded eigenfunction}
\end{equation}
}
\item The unfolded function $\unfold\varphi$ is an eigenfunction corresponding
to the eigenvalue $2\lambda$ and is given by
\begin{equation}
\begin{array}{c}
\unfold\varphi=\underset{k^{2}+l^{2}=\lambda}{\sum}\alpha_{k,l}\varphi_{_{U_{\Q}\left(k,l\right)}},\\
\\
\textrm{where\,\,\,\, }U_{\Q}\left(k,l\right)=\left(k+l,k-l\right).
\end{array}\label{eq:formula unfolded}
\end{equation}

\end{enumerate}
\end{lem}
\begin{proof}
First consider the case that $\varphi=\varphi_{k,l}$. A simple calculation
of $\fold\varphi_{k,l}$ and $\unfold\varphi_{k,l}$ involving the
trigonometric identity 
\[
2\cos\left(\alpha\right)\cos\left(\beta\right)=\cos\left(\alpha+\mbox{\ensuremath{\beta}}\right)+\cos\left(\alpha-\beta\right),
\]
yields 
\[
\fold\varphi_{k,l}=\varphi_{F_{\Q}\left(k,l\right)}
\]
and
\[
\unfold\varphi_{k,l}=\varphi_{U_{\Q}(k,l)}.
\]
To conclude that $\fold\varphi_{k,l}$ and $\unfold\varphi_{k,l}$
are eigenfunctions we need to verify that $F_{Q}\left(k,l\right),~U_{Q}\left(k,l\right)\in\Q$.
This is obvious for $U_{Q}\left(k,l\right)$, and as for $F_{Q}\left(k,l\right)$
we use that $\lambda_{k,l}\in\espec{\trngle}$ implies $\left(k,l\right)\in\E$
and thus $\left(\frac{k+l}{2},\frac{k-l}{2}\right)\in\Q$. The last
part of the claim is that $\fold\varphi_{k,l}$ and $\unfold\varphi_{k,l}$
correspond to eigenvalues $\frac{1}{2}\lambda_{k,l}$ and $2\lambda_{k,l}$.
Indeed we have

\noindent 
\[
\left\Vert F_{Q}\left(k,l\right)\right\Vert ^{2}=\frac{1}{2}\left(k^{2}+l^{2}\right)=\frac{1}{2}\lambda_{k,l},
\]
and
\[
\left\Vert U_{Q}\left(k,l\right)\right\Vert ^{2}=2\left(k^{2}+l^{2}\right)=2\lambda_{k,l}.
\]
Finally, using the linearity of $\fold$ and $\unfold$ we conclude
that the claim holds for 
\[
\varphi=\underset{k^{2}+l^{2}=\lambda}{\sum}\alpha_{k,l}\varphi_{k,l}.
\]

\end{proof}
The last lemma allows for a useful characterization of all eigenvalues.
\begin{cor}
\begin{singlespace}
~\label{cor:eiganvalue_is_odd_times_power_of_two}\end{singlespace}

\begin{enumerate}
\item Let $0\neq\lambda\in\spec{\trngle}$. Then there exist unique $\oval\in\ospec{\trngle}$
and $k\in\N_{0}$ such that $\lambda=2^{k}\oval$. Furthermore, $\mult\left(\lambda\right)=\mult\left(\oval\right)$.
\item Let $\oval\in\ospec{\trngle}$ and $k\in\N_{0}$. Then $2^{k}\oval\in\spec{\trngle}$.
\end{enumerate}
\end{cor}
\begin{proof}
\begin{singlespace}
~\end{singlespace}

\begin{enumerate}
\item Let $0\neq\lambda\in\spec{\trngle}$. As $\sigma\left(\trngle\right)\subseteq\mathbb{N}_{0}$
we can write uniquely 
\begin{equation}
\lambda=2^{k}\oval\,\,\text{with}~~\,k\in\mathbb{N}_{0},\,\,\oval\in\mathbb{N}_{0}\,\,\textrm{odd}.\label{eq:eigenvalue_presented_in_terms_of_odd_evalue}
\end{equation}
In order to show $\oval\in\ospec{\trngle}$ consider $\varphi$ to
be an eigenfunction of $\lambda=2^{k}\oval$. By Lemma \ref{lem:fold_unfold_efunc}
it follows that $\fold^{k}\varphi$ is an eigenfunction and its corresponding
eigenvalue equals $2^{-k}\lambda=\oval$. Hence, $\oval\in\spec{\trngle}$,
but as it is odd we further have $\oval\in\ospec{\trngle}$. The equality
of multiplicities of $\lambda$ and $\oval$ arises as $\fold^{k}$
is a linear isomorphism (its inverse is $\unfold^{k}$) from the eigenspace
of $\lambda$ to the eigenspace of $\oval$.
\item Let $\oval\in\ospec{\trngle}$ and $k\in\N_{0}$. By Lemma \ref{lem:fold_unfold_efunc},
$\unfold^{k}$ maps an eigenfunction of $\oval$ to an eigenfunction
of $2^{k}\oval$.
\end{enumerate}
\end{proof}
The last corollary implies that the spectrum has the following hierarchical
structure
\begin{equation}
\sigma\left(\mathcal{D}\right)=\bigsqcup_{k=0}^{\infty}\left(\bigcup_{\oval\in\ospec{\trngle}}\left\{ 2^{k}\oval\right\} \right)=\bigsqcup_{k=0}^{\infty}\left(\bigcup_{\left(m,n\right)\in\O}\left\{ \lambda_{U_{\Q}^{k}\left(m,n\right)}\right\} \right),\label{eq:hirarchical_structure_of_spectrum}
\end{equation}
where the second equality follows since 
\[
\oval\in\ospec{\trngle}~~~\iff~~~\oval=\lambda_{m,n}~\textrm{s.t. }\left(m,n\right)\in\O~~~\iff~~~2^{k}\oval=\lambda_{U_{\Q}^{k}\left(m,n\right)}\textrm{ s.t. }\left(m,n\right)\in\O.
\]
 We will use this structure to divide the spectrum into three subsets,
and rule out separately the Courant-sharpness of the eigenvalues in
each of those subsets (the three parts of Proposition \ref{prop:rulling_all_courant_sharp_traingle}).

\begin{singlespace}
In the following we will show for a given $k\in\mathbb{N}_{0}$ and
any $\oval\in\ospec{\trngle}$ that the eigenfunctions corresponding
to the eigenvalue $2^{k}\oval$ all vanish on a specific ($k$-dependent)
subset of $\trngle$.
\end{singlespace}
\begin{defn}
\begin{singlespace}
Let $A\subseteq\mathcal{D}$, We define the \emph{unfold} of $A$
as
\begin{equation}
U\left(A\right)=\left\{ \left(x,y\right)\in\hftrngle\,\at\,F\left(x,y\right)\in A\right\} \cup\left\{ \left(x,y\right)\in\trngle\backslash\hftrngle\,\at\,F\circ R\left(x,y\right)\in A\right\} .\label{eq:unfolding_of_a_set}
\end{equation}
For $k\geq0$ we define the \emph{k-frame }as (see Figure \ref{fig:four_unfoldings-1})
\begin{equation}
S^{(0)}=\line\,\,\,\,\textrm{and \,\,\,}\forall k\geq1,~S^{(k)}=U^{k}\left(\line\right).\label{eq:k_frame}
\end{equation}
\end{singlespace}

\end{defn}
\begin{figure}
~~~~~~~~~~~~~~~~~~~~~~~~~~~~~~~~~~~~~~~~~~~~~~~~~~~~~~~~~~~~~~~~~~~~~~~~~~~~~~~~~~~~~~~~~~~~~~~~~~~~~~~~~~~~~~~~~~~~~~~~~~~\includegraphics[scale=0.5]{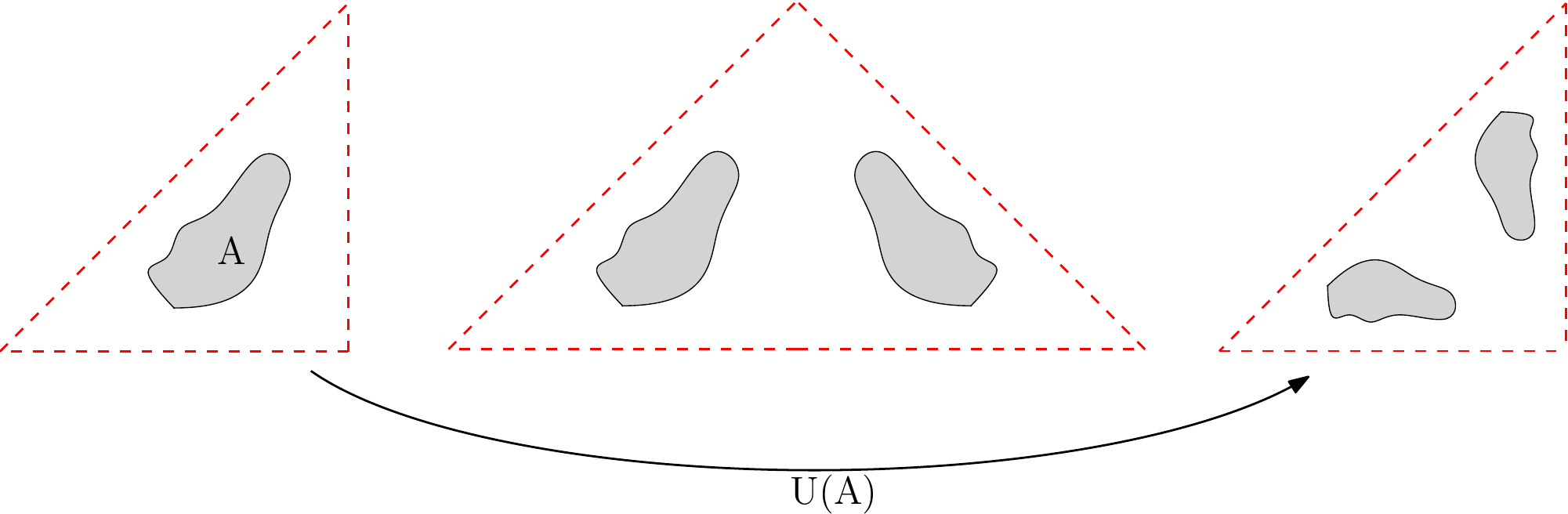}

\protect\caption{Unfolding of a set}
\end{figure}
\begin{figure}
\medskip{}

\begin{minipage}[c]{0.3\columnwidth}%
\emph{\uline{0-frame}}

\includegraphics[scale=0.35]{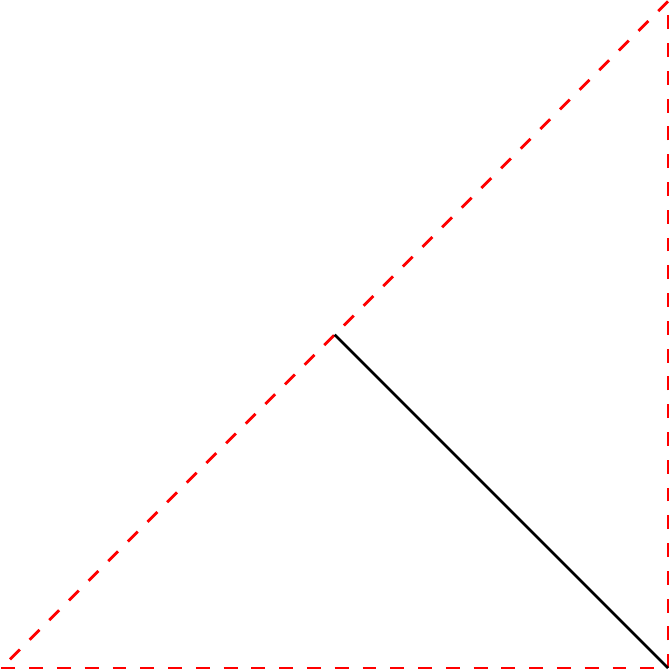}\hspace*{\fill}$\begin{array}{c}
\longrightarrow\\
_{U}\\
\\
\\
\end{array}$\hspace*{\fill}%
\end{minipage}%
\begin{minipage}[c]{0.3\columnwidth}%
\emph{\uline{1-frame}}

\includegraphics[scale=0.35]{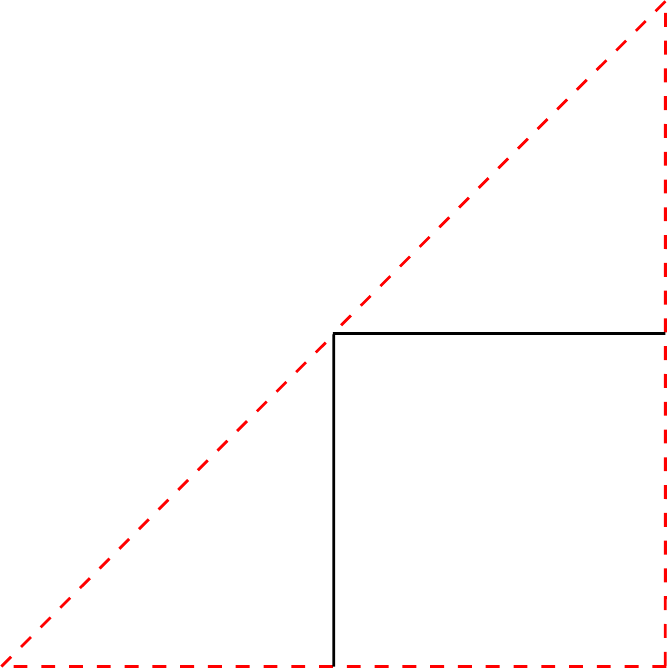}\hspace*{\fill}$\begin{array}{c}
\longrightarrow\\
_{U}\\
\\
\\
\end{array}$\hspace*{\fill}%
\end{minipage}%
\begin{minipage}[c]{0.3\columnwidth}%
\emph{\uline{2-frame}}

\includegraphics[scale=0.35]{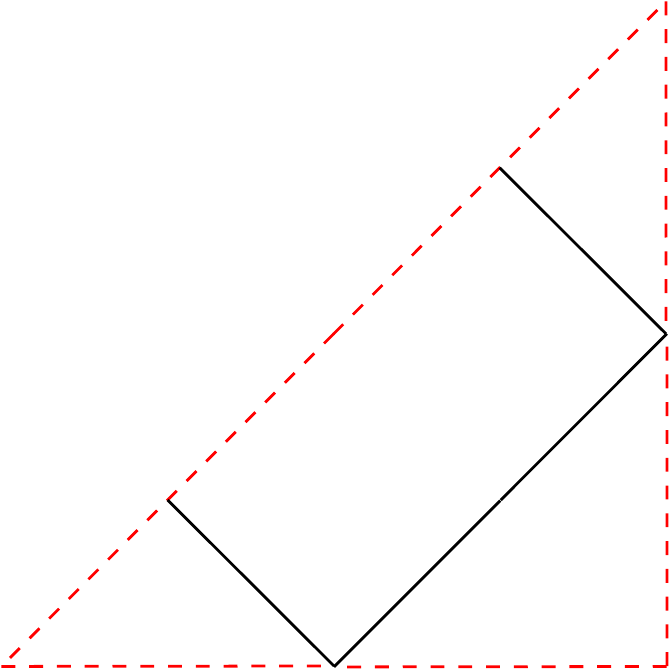}\hspace*{\fill}$\begin{array}{c}
\longrightarrow\\
_{U}\\
\\
\\
\end{array}$\hspace*{\fill}%
\end{minipage}

\vspace{4mm}

\begin{minipage}[c]{0.35\columnwidth}%
\emph{\uline{3-frame}}

\includegraphics[scale=0.35]{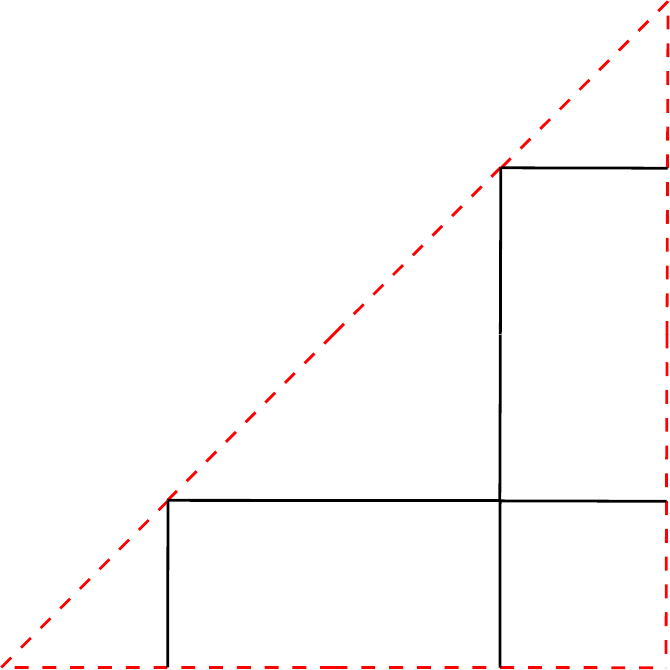}\hspace*{\fill}$\begin{array}{c}
\longrightarrow\\
_{U}\\
\\
\\
\end{array}$\hspace*{\fill}%
\end{minipage}%
\begin{minipage}[c]{0.35\columnwidth}%
\emph{\uline{4-frame}}

\includegraphics[scale=0.35]{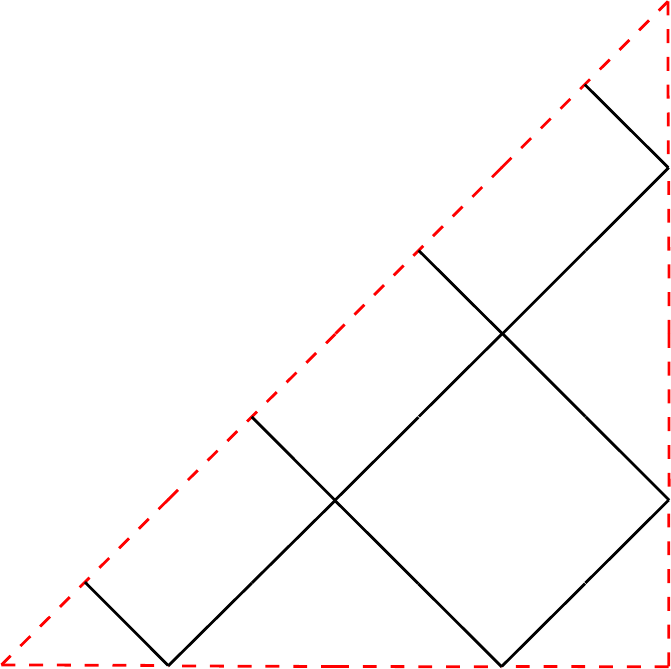}$\begin{array}{c}
\\
\\
\\
\\
\end{array}$%
\end{minipage}

\protect\caption{The first five k-frames indicated by{\small{} black solid lines.}\label{fig:four_unfoldings-1}}
\end{figure}

\begin{prop}
\label{prop:k_frame_vanishing} Let $k\in\mathbb{N}_{0}\,\textrm{and }\oval\in\ospec{\trngle}$
then any eigenfunction corresponding to $2^{k}\oval$ vanishes on
the k-frame.\end{prop}
\begin{proof}
Let $\oval\in\ospec{\trngle}$, we shall prove the claim for $2^{k}\oval$
by induction on $k$. For $k=0$, by Lemma \ref{lem:eigenfunction_symmetry}
we get that any eigenfunction $\varphi$ of $\oval$ is anti-symmetric
with respect to $\line$ and therefore
\[
\left.\varphi\right|_{_{S^{\left(0\right)}}}=0.
\]
Next assume that the claim holds for $k-1$, and let $\varphi$ be
an eigenfunction corresponding to $2^{k}\oval$. By Lemma \ref{lem:fold_unfold_efunc}
we have that $F\varphi$ is an eigenfunction corresponding to the
eigenvalue $2^{k-1}\oval$ and hence
\[
\left.\fold\varphi\right|_{_{S^{\left(k-1\right)}}}=0.
\]
Now let $\left(x,y\right)\in S^{(k)}$. Since $S^{(k)}=U\left(S^{\left(k-1\right)}\right)$
we get that $F\left(x,y\right)\in S^{(k-1)}$ or \\
$F\circ R\left(x,y\right)\in S^{(k-1)}$ (see \eqref{eq:unfolding_of_a_set}).
If $F\left(x,y\right)\in S^{(k-1)}$ then 
\[
\varphi\left(x,y\right)=\varphi\circ U\circ F\left(x,y\right)=\fold\varphi\left(F\left(x,y\right)\right)=0.
\]
If $F\circ R\left(x,y\right)\in S^{(k-1)}$ we note that $\varphi$
is symmetric as $2^{k}\oval$ is even and hence
\[
\varphi\left(x,y\right)=\varphi\left(R\left(x,y\right)\right)=\varphi\circ U\circ F\left(R\left(x,y\right)\right)=\fold\varphi\left(F\circ R\left(x,y\right)\right)=0.
\]
Therefore, $\left.\varphi\right|_{_{S^{\left(k\right)}}}=0.$
\end{proof}

\section{Proof of Theorem \ref{thm:CourantSharp_triangle} \label{sec:proof_of_triangle}}

Using the definitions of Section \ref{sec:folding_structure_triangle},
the set of Courant-sharp eigenvalues according to Theorem \ref{thm:CourantSharp_triangle}
can be written as

\textbf{
\begin{equation}
\mathcal{C}=\{0\}\cup\left\{ \lambda_{U_{Q}^{k}\left(1,0\right)}\right\} _{k=0}^{3}.\label{eq:CS_eigenvalues}
\end{equation}
}

We divide the remaining eigenvalues, $\sigma\left(\mathcal{D}\right)\backslash\mathcal{C}$,
into three subsets and rule out their Courant-sharpness by the following
proposition.
\begin{prop}
\label{prop:rulling_all_courant_sharp_traingle}The eigenvalues of
each of the following sets are not Courant-sharp.
\begin{enumerate}
\item \label{enu:prop_zeroth_unfoldings_not_CS} 
\[
\Lambda^{(1)}:=\underset{\begin{array}{c}
\left(m,n\right)\in\O\backslash\left\{ \left(1,0\right)\right\} \end{array}}{\bigcup}\left\{ \lambda_{m,n}\right\} 
\]
 
\item \label{enu:prop_almost all high unfoldings are non courant sharp}
\[
\Lambda^{(2)}:=\bigsqcup_{k=1}^{\infty}\left(\underset{(m,n)\in\O:\,n\neq0}{\bigcup}\left\{ \lambda_{U_{\Q}^{k}\left(m,n\right)}\right\} \right)
\]

\item \label{enu:prop_The-eigenvalues-of U(m,0) not courant sharp} 
\[
\Lambda^{(3)}:=\bigsqcup_{k=1}^{\infty}\left(\underset{(m,0)\in\O\backslash\left\{ \left(1,0\right)\right\} }{\bigcup}\left\{ \lambda_{U_{\Q}^{k}\left(m,0\right)}\right\} \right)\bigcup\left\{ \lambda_{U_{\Q}^{k}\left(1,0\right)}\right\} _{k=4}^{\infty}
\]

\end{enumerate}
\end{prop}
The proofs of the three parts of the proposition are essentially different
and each appears in a designated subsection.

\subsection{Proving Proposition \ref{prop:rulling_all_courant_sharp_traingle},\eqref{enu:prop_zeroth_unfoldings_not_CS}.}

We start by providing some additional constructions, needed for the
proof.
\begin{defn}
We define the following subsets of the lattice $\Q$.\end{defn}
\begin{enumerate}
\item For $0\leq\lambda\in\mathbb{R}$ we define
\begin{equation}
\begin{array}{c}
\mathcal{O}\left(\lambda\right)=\Q\left(\lambda\right)\cap\O,\\
\\
\mathcal{E}\left(\lambda\right)=\Q\left(\lambda\right)\cap\E.
\end{array}\label{eq:odd_even_smaller_then_lambda}
\end{equation}

\item Let $A\subseteq\Q$, define $\underrightarrow{\partial}A$ to be the
set of points in $A$ such that their right neighbor is outside $A$,
i.e:
\begin{equation}
\underrightarrow{\partial}A=\left\{ (m,n)\in A\,|\,(m+1,n)\notin A\right\} .\label{eq:right_boundary_points}
\end{equation}

\end{enumerate}
\textbf{}
\begin{figure}[H]
\includegraphics[scale=0.7]{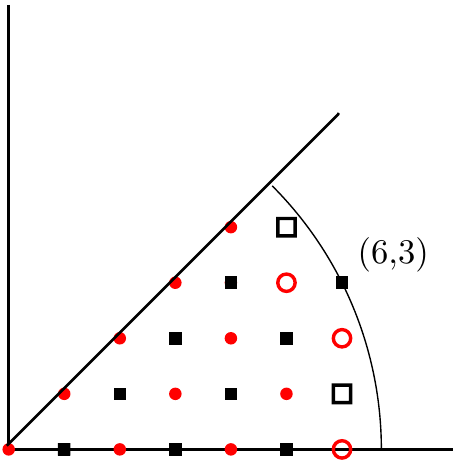}

\textbf{\protect\caption{\textbf{\label{fig:lattice points and boundary points}}{\small{}The
red disks correspond to $\mathcal{E}\left(\lambda_{6,3}\right)$ and
the black squares correspond to $\protect\O\left(\lambda_{6,3}\right)$.
The empty disks and squares correspond to $\protect\underrightarrow{\partial}\protect\Q\left(\lambda_{6,3}\right)$.} }
}
\end{figure}

We consider the following auxiliary eigenvalue problem on $\hftrngle$
with mixed Dirichlet-Neumann boundary conditions 
\begin{equation}
-\Delta\varphi=\lambda\varphi,\qquad\varphi\Big\vert_{\line}=0\,,\,\,\frac{\partial\varphi}{\partial\vec{n}}\Big\vert_{\partial\left(\hftrngle\right)\backslash L}=0.\label{eq:dnn_problem-1}
\end{equation}

Denote the corresponding spectrum by $\spec{\hftrngle}$ and the spectral
counting functions of \eqref{eq:dnn_problem-1} by $\underline{\tilde{N}}\left(\lambda\right),\,\,\tilde{N}\left(\lambda\right)$
as in \eqref{eq:lower_counting}-\eqref{eq:counting}.
\begin{lem}
Let $0\leq\lambda\in\R$ then 
\begin{equation}
\underline{\tilde{N}}\left(\lambda\right)=\left|\mathcal{O}\left(\lambda\right)\right|.\label{eq:number of odd_qn equal to the dnn counting}
\end{equation}
\end{lem}
\begin{proof}
Note that if $\lambda_{p,q}\in\ospec{\trngle}$ and $\tilde{\varphi}$
is any of its eigenfunctions, then by Lemma \ref{lem:eigenfunction_symmetry}
it follows that $\tilde{\varphi}\at_{\hftrngle}$ is an eigenfunction
of \eqref{eq:dnn_problem-1} which means that $\lambda_{p,q}\in\spec{\hftrngle}$.
This motivates to consider the mapping 
\[
\varPhi:\mathcal{O}\left(\lambda\right)\rightarrow\left\{ \left.\tilde{\lambda}\in\spec{\hftrngle}\right|\,\tilde{\lambda}<\lambda\right\} 
\]
\[
\varPhi:\left(p,q\right)\longmapsto\lambda_{p,q},
\]
where we note that the set $\left\{ \left.\tilde{\lambda}\in\spec{\hftrngle}\right|\,\tilde{\lambda}<\lambda\right\} $
contains each eigenvalue as many times as its multiplicity in $\spec{\hftrngle}$.
Showing that $\varPhi$ is a bijection proves the lemma. To show that
$\varPhi$ is onto, let $\tilde{\lambda}\in\left\{ \left.\tilde{\lambda}\in\spec{\hftrngle}\right|\,\tilde{\lambda}<\lambda\right\} $
and extend one of its corresponding eigenfunctions $\tilde{\varphi}$
anti-symmetrically along $\line$, i.e. consider 
\[
\varphi\left(x,y\right)=\begin{cases}
\tilde{\varphi}\left(x,y\right) & \left(x,y\right)\in\hftrngle\\
-\tilde{\varphi}\circ R\left(x,y\right) & \left(x,y\right)\in\mathcal{D}\backslash\hftrngle
\end{cases}.
\]
By the reflection principle (see \cite{herbrich_arX14} for example),
$\varphi$ is an odd eigenfunction of \eqref{eq:eigen_prob}. By Lemma
\ref{lem:eigenfunction_symmetry} we deduce $\tilde{\lambda}\in\ospec{\trngle}$,
and so there exists $\left(p,q\right)\in\O\left(\lambda\right)$ such
that $\lambda_{p,q}=\tilde{\lambda}.$ To show that $\varPhi$ is
an injection, take $\left(p_{1},q_{1}\right)\neq\left(p_{2},q_{2}\right)$
and observe that the eigenfunctions $\varphi_{p_{1},q_{1}}$, $\varphi_{p_{2},q_{2}}$
are linearly independent and anti-symmetric and hence $\varphi_{p_{1},q_{1}}\at_{\hftrngle},\varphi_{p_{2},q_{2}}\at_{\hftrngle}$
are linearly independent and are eigenfunctions of \eqref{eq:dnn_problem-1},
which means that $\varPhi\left(p_{1},q_{1}\right)\neq\varPhi\left(p_{2},q_{2}\right)$.
\end{proof}
We are now able to prove Proposition \ref{prop:rulling_all_courant_sharp_traingle},\eqref{enu:prop_zeroth_unfoldings_not_CS}.
\begin{proof}[Proof of Proposition \ref{prop:rulling_all_courant_sharp_traingle},\eqref{enu:prop_zeroth_unfoldings_not_CS}.]
Let $\lambda_{m,n}$ be such that $\left(m,n\right)\in\O\backslash\left\{ \left(1,0\right)\right\} $
and $\varphi$ be any eigenfunction corresponding to $\lambda_{m,n}$,
Lemma \ref{lem:eigenfunction_symmetry} gives
\begin{equation}
\nu\left(\varphi\right)=2\cdot\nu\left(\left.\varphi\right|_{_{\hftrngle}}\right).\label{eq:nodal_domains_antisymm}
\end{equation}
Since $\left.\varphi\right|_{_{\hftrngle}}$ is an eigenfunction of
\eqref{eq:dnn_problem-1} with an eigenvalue $\lambda_{m,n}$, we
get by Courant's nodal theorem \cite{Cou_ngwgmp23} that
\begin{equation}
\nu\left(\left.\varphi\right|_{_{\hftrngle}}\right)\leq\tilde{N}\left(\lambda_{m,n}\right),\label{eq:DNN_eigenvals_counting}
\end{equation}
so that
\[
\nu\left(\varphi\right)\leq2\cdot\tilde{N}\left(\lambda_{m,n}\right)
\]
 and by \eqref{eq:number of odd_qn equal to the dnn counting} we
obtain
\begin{equation}
\nu\left(\varphi\right)\leq2\cdot\left(\O\left(\lambda_{m,n}\right)+1\right).\label{eq:dnn courant bound on trngle nodal domains}
\end{equation}
 Next, observe that the following mapping
\begin{equation}
\begin{array}{c}
B:\,\,\mathcal{O}\left(\lambda_{m,n}\right)\rightarrow\mathcal{E}\left(\lambda_{m,n}\right)\backslash\left(\underrightarrow{\partial}\Q\left(\lambda_{m,n}\right)\cap\mathcal{E}\right)\\
\\
B(p,q)=(p-1,q),
\end{array}\label{eq:boundary_map}
\end{equation}
is a bijection (see Figure \ref{fig:lattice points and boundary points})
and thus we obtain
\begin{equation}
\left|\mathcal{O}\left(\lambda_{m,n}\right)\right|=\left|\mathcal{E}\left(\lambda_{m,n}\right)\right|-\left|\underrightarrow{\partial}\Q\left(\lambda_{m,n}\right)\cap\mathcal{E}\right|.\label{eq:even_bigger_then_odd}
\end{equation}
We now have,
\begin{equation}
\nu\left(\varphi\right)\underbrace{\leq}_{\textrm{ \eqref{eq:dnn courant bound on trngle nodal domains} }}2\left(\left|\mathcal{O}\left(\lambda_{m,n}\right)\right|+1\right)\underbrace{=}_{\eqref{eq:even_bigger_then_odd}}\left|\mathcal{O}\left(\lambda_{m,n}\right)\right|+\left|\mathcal{E}\left(\lambda_{m,n}\right)\right|+1+\left(1-\left|\underrightarrow{\partial}\Q\left(\lambda_{m,n}\right)\cap\mathcal{E}\right|\right).\label{eq:inequality of nodal domains of an odd eigenfunction by lattice}
\end{equation}
Note that 
\[
\Q\left(\lambda_{m,n}\right)=\mathcal{E}\left(\lambda_{m,n}\right)\bigsqcup\mathcal{O}\left(\lambda_{m,n}\right),
\]
 hence 
\begin{equation}
\underline{N}\left(\lambda_{m,n}\right)=\left|\mathcal{O}\left(\lambda_{m,n}\right)\right|+\left|\mathcal{E}\left(\lambda_{m,n}\right)\right|.\label{eq:odd_plus_even}
\end{equation}
Combining \eqref{eq:inequality of nodal domains of an odd eigenfunction by lattice}
with \eqref{eq:odd_plus_even} we get
\begin{equation}
\nu\left(\varphi\right)\leq N\left(\lambda_{m,n}\right)+\left(1-\left|\underrightarrow{\partial}\Q\left(\lambda_{m,n}\right)\cap\mathcal{E}\right|\right).\label{eq:nodal deficiecy of odd eigenvalues}
\end{equation}
Therefore in order to rule out the Courant-sharpness of $\lambda_{m,n}$
we only require that 
\begin{equation}
\left|\underrightarrow{\partial}\Q\left(\lambda_{m,n}\right)\cap\mathcal{E}\right|>1.\label{eq:boundary_has_at_least_two_points}
\end{equation}
Indeed, since $\left(m,n\right)\in\O$ we get $\left(m-1,n\right)\in\underrightarrow{\partial}\Q\left(\lambda_{m,n}\right)\cap\mathcal{E}$
and we are left with finding one more point $\left(p,q\right)\in\underrightarrow{\partial}\Q\left(\lambda_{m,n}\right)\cap\mathcal{E}$.
As we consider $\left(m,n\right)\in\O\backslash\left\{ \left(1,0\right)\right\} $,
a simple calculation shows that if $n\geq1$ then $\left(m,n-1\right)\in\underrightarrow{\partial}\Q\left(\lambda_{m,n}\right)\cap\mathcal{E}$
and if $n=0$ we have $\left(m-1,2\right)\in\underrightarrow{\partial}\Q\left(\lambda_{m,0}\right)\cap\mathcal{E}$.\end{proof}
\begin{rem*}
It is easy to see that the last argument does not work for $\left(m,n\right)=\left(1,0\right)$.
Indeed we show later that this is a Courant-sharp eigenvalue (Lemma
\ref{lem:cournat sharpness of the eigenvals}).
\end{rem*}

\subsection{Proving Proposition \ref{prop:rulling_all_courant_sharp_traingle},\eqref{enu:prop_almost all high unfoldings are non courant sharp}.}

The $k$-frame structure divides the triangle into $k$-dependent
number of subdomains. This is defined below and is used in the proofs
of the current subsection.
\begin{defn}
\label{def:k_frame_partition} Define the $k$-frame partition as
\[
\P^{\left(k\right)}\coloneqq\trngle^{\circ}\backslash S^{(k)}=\bigsqcup_{i=1}^{M(k)}\trngle_{i}^{(k)},
\]
where $\left\{ \trngle_{i}^{\left(k\right)}\right\} _{i=1}^{M\left(k\right)}$denote
the subdomains of this partition and $M\left(k\right)$ is their number. 

\noindent Consider the following eigenvalue problems with the boundary
conditions induced by the $k$-frame
\begin{equation}
-\Delta\varphi=\lambda\varphi,\qquad\varphi\Big\vert_{S^{(k)}\bigcap\partial\trngle_{i}^{(k)}}=0\,,\,~~~~~~~\quad\frac{\partial\varphi}{\partial\vec{n}}\Big\vert_{\partial\trngle\bigcap\partial\trngle_{i}^{(k)}}=0.\label{eq:subdomains eigenvalue problem}
\end{equation}
We denote the corresponding spectra by $\spec{\trngle_{i}^{(k)}}$
and define the corresponding spectral counting functions, and multiplicities
\[
\nup_{i}^{\left(k\right)}\left(\lambda\right)\,,\,\,\underline{N}_{i}^{\left(k\right)}\left(\lambda\right),\,\,N_{i}^{\left(k\right)}\left(\lambda\right),\,\,\mult_{i}^{(k)}\left(\lambda\right),
\]
 as in \eqref{eq:upper_counting}-\eqref{eq:multiplicity}.
\end{defn}
Next we bring two lemmata, the second of which provides necessary
conditions for an eigenvalue to be Courant-sharp.
\begin{lem}
\label{lem:degeneracy of subdomain} Let $k\in\mathbb{N}_{0}$ and
$\oval\in\ospec{\mathcal{D}}$. We have that
\begin{equation}
\mult\left(2^{k}\oval\right)\leq\mult_{i}^{(k)}\left(2^{k}\oval\right)\ \ i\in\left\{ 1,..,M\left(k\right)\right\} .\label{eq:degeneracy}
\end{equation}
\end{lem}
\begin{proof}
Let $\lambda=2^{k}\oval$and let 
\[
\B=\left\{ \ef_{1},..,\ef_{\mult\left(\lambda\right)}\right\} 
\]
 be a basis for the eigenspace of $\lambda$. Let $i\in\left\{ 1,..,M\left(k\right)\right\} $,
by Proposition \ref{prop:k_frame_vanishing} we have that $\B'=\left\{ \left.\ef_{1}\right|_{\trngle_{i}^{(k)}},..,\left.\ef_{\mult\left(\lambda\right)}\right|_{\trngle_{i}^{(k)}}\right\} $
are eigenfunctions of \eqref{eq:subdomains eigenvalue problem} on
the domain $\trngle_{i}^{(k)}$. Assume by contradiction that the
set $\B'$ turns out to be linearly dependent, then we have scalars
$\alpha_{l}\in\mathbb{R}$ not all zero such that
\[
\sum_{l}\left.\alpha_{l}\ef_{l}\right|_{\trngle_{i}^{(k)}}\equiv0.
\]
But then the eigenfunction $\sum_{l}\alpha_{l}\ef_{l}$ of \eqref{eq:eigen_prob}
vanishes on the open subset $\trngle_{i}^{(k)}$, and by the unique
continuation property \cite{Aronszajn_jmpa57} we obtain that
\[
\sum_{l}\alpha_{l}\ef_{l}\equiv0,
\]
contradicting the linear independence of $\B$. Thus it follows that
the dimension of the eigenspace that corresponds to $\lambda\in\spec{\trngle_{i}^{(k)}}$
is at least $\mult\left(\lambda\right)$ .\end{proof}
\begin{rem*}
The strict inequality in \eqref{eq:degeneracy} may indeed occur.
This can be demonstrated by applying Corollary \ref{cor:eiganvalue_is_odd_times_power_of_two},(1)
and Lemma \ref{lem:high unfoldings degenerate eigenvalues} to some
simple eigenvalue $\lambda_{m,n}\in\ospec{\trngle}$ such that $n\neq0.$
\begin{lem}
\label{lem:CS_implies_simplicity} Let $k\in\mathbb{N}_{0}$ and $\oval\in\ospec{\mathcal{D}}$.
If $2^{k}\oval$ is a Courant-sharp eigenvalue of $\trngle$ then
both of the following hold\end{lem}
\begin{enumerate}
\item \emph{\label{enu:lem_CS_implies_simplicity_1}The eigenvalue $2^{k}\oval$
is a simple eigenvalue in $\spec{\trngle_{i}^{(k)}}$ for all $i\in\left\{ 1,..,M\left(k\right)\right\} $.}
\item \emph{\label{enu:lem_CS_implies_simplicity_2}The eigenvalue $2^{k}\oval$
is a simple eigenvalue in $\spec{\trngle}$.}
\end{enumerate}
\end{rem*}
\begin{proof}
Let $\varphi$ be a Courant-sharp eigenfunction of $2^{k}\oval\in\spec{\trngle}$,
then
\begin{equation}
N\left(2^{k}\oval\right)=\nu\left(\varphi\right).\label{eq:number of domains for courantsharp function}
\end{equation}
By Proposition \ref{prop:k_frame_vanishing} we have
\begin{equation}
\nu\left(\varphi\right)=\sum_{i=1}^{M(k)}\nu\left(\left.\varphi\right|_{_{\trngle_{i}^{(k)}}}\right).\label{eq:nodal domains of sum over the partition by k-frame}
\end{equation}
For all $i\in\left\{ 1,...,M(k)\right\} $ we have that $\left.\varphi\right|_{_{\trngle_{i}^{(k)}}}$
is an eigenfunction of the eigenvalue problem \eqref{eq:subdomains eigenvalue problem},
and therefore by Courant's nodal theorem we have
\begin{equation}
\sum_{i=1}^{M(k)}\nu\left(\left.\varphi\right|_{_{\trngle_{i}^{(k)}}}\right)\leq\sum_{i=1}^{M(k)}N_{i}^{(k)}\left(2^{k}\oval\right).\label{eq:courant thm on the partition}
\end{equation}
Rewriting the right-hand side of \eqref{eq:courant thm on the partition}
and combining \eqref{eq:number of domains for courantsharp function},
\eqref{eq:nodal domains of sum over the partition by k-frame} we
arrive at 
\[
N\left(2^{k}\oval\right)\leq\sum_{i=1}^{M(k)}\nup_{i}^{(k)}\left(2^{k}\oval\right)+\sum_{i=1}^{M(k)}\left[N_{i}^{(k)}\left(2^{k}\oval\right)-\nup_{i}^{(k)}\left(2^{k}\oval\right)\right].
\]
If we consider the eigenvalue problem on $\bigcup_{i}\trngle_{i}^{(k)}$
and use the variational principle to compare with the eigenvalue problem
on $\trngle$ we obtain 
\begin{equation}
\sum_{i=1}^{M(k)}\nup_{i}^{(k)}\left(2^{k}\oval\right)\leq\nup\left(2^{k}\oval\right).\label{eq:variational principle}
\end{equation}
The conclusion above appears for example in \cite{CourantHilbert_volume1},
page 408, Theorem 2 for Dirichlet boundary conditions. Having Neumann
boundary conditions, as in our case, brings to the same conclusion.

It follows that
\begin{equation}
\sum_{i=1}^{M(k)}\left[\nup_{i}^{(k)}\left(2^{k}\oval\right)-N_{i}^{(k)}\left(2^{k}\oval\right)\right]\leq\nup\left(2^{k}\oval\right)-N\left(2^{k}\oval\right).\label{eq:sum of subdomains counting small the domain counting}
\end{equation}
 By Lemma \ref{lem:degeneracy of subdomain} we have 
\begin{equation}
\nup\left(2^{k}\oval\right)-N\left(2^{k}\oval\right)\leq\nup_{i}^{(k)}\left(2^{k}\oval\right)-N_{i}^{(k)}\left(2^{k}\oval\right)\ \ \forall i\in\left\{ 1,..,M\left(k\right)\right\} .\label{eq:inequality of counting function d-1}
\end{equation}
Plugging this in \eqref{eq:sum of subdomains counting small the domain counting}
we get 
\[
\sum_{l=1}^{M(k)}\left[\nup_{l}^{(k)}\left(2^{k}\oval\right)-N_{l}^{(k)}\left(2^{k}\oval\right)\right]\leq\nup_{i}^{(k)}\left(2^{k}\oval\right)-N_{i}^{(k)}\left(2^{k}\oval\right)\ \ \forall i\in\left\{ 1,..,M\left(k\right)\right\} ,
\]
and as $M(k)\geq2$ , it has to be that 
\[
\nup_{i}^{(k)}\left(2^{k}\oval\right)-N_{i}^{(k)}\left(2^{k}\oval\right)=0\ \ \forall i\in\left\{ 1,..,M\left(k\right)\right\} ,
\]
 which proves the first part of the lemma. The second part follows
immediately from a combination of the first part with Lemma \ref{lem:degeneracy of subdomain}.\end{proof}
\begin{rem*}
The first claim of Lemma \ref{lem:CS_implies_simplicity} is not restricted
to the domains dealt with so far and indeed appears in a more general
form in Corollary 3.5(ii) of \cite{nodal_thms_a_la_courant}, where
it is proven for domains with Dirichlet boundary conditions. Yet,
the second claim of Lemma \ref{lem:CS_implies_simplicity} does not
hold for arbitrary domains, as it is based on the inequality \eqref{eq:degeneracy}
which is not satisfied in general.
\end{rem*}
~
\begin{rem*}
Lemma \ref{lem:CS_implies_simplicity},(2) may be also obtained as
a direct corollary of Lemma \ref{lem:appendix_nodal_deficiency} (appears
in Appendix \ref{sec:appendix_nodal_deficiency}) and Corollary \ref{cor:eiganvalue_is_odd_times_power_of_two},(1).
In fact, Lemma \ref{lem:appendix_nodal_deficiency} is a generalization
of Lemma \ref{lem:CS_implies_simplicity},(2).
\end{rem*}
With Lemma \ref{lem:CS_implies_simplicity} in hand, it is now possible
to rule out the Courant-sharpness of many more eigenvalues. This is
done by applying the lemma to the following particular subdomains
of the triangle.
\begin{defn}
We define the following subdomains of the $k$-frame partition:\end{defn}
\begin{enumerate}
\item A square subdomain $\sdom\in\P^{(1)}$ expressed by
\begin{equation}
\sdom=(\nicefrac{\pi}{2},\pi)\times(0,\nicefrac{\pi}{2}).\label{eq:square subdomain}
\end{equation}
\begin{figure}[H]
\begin{minipage}[c][1\totalheight][t]{0.3\columnwidth}%
\qquad{}\quad{}\includegraphics[scale=0.5]{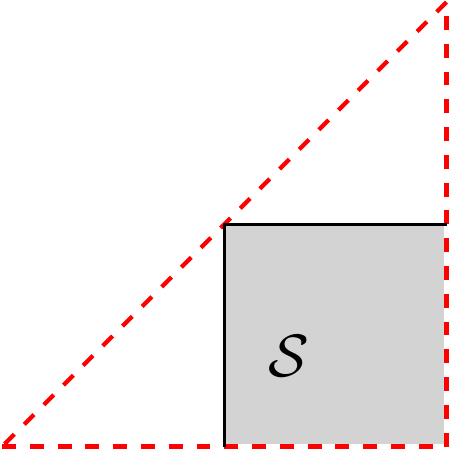}%
\end{minipage}\protect\caption{\label{fig:square}The subdomain $\protect\sdom$ as it appears in
the $1$-frame partition, $\protect\P_{1}$.}
\end{figure}

\item Rectangular subdomains $\recdom k\in\P^{(k)}\ ,\forall k\geq2$, expressed
recursively by\textbf{ }
\begin{equation}
\recdom 2=\left(U\left(\overline{\sdom}\right)\right)^{\circ}\textrm{\ \ and\ \ }\recdom k=\left(\left\{ \left(x,y\right)\in\trngle\at\,F\left(x,y\right)\in\overline{\recdom{k-1}}\right\} \right)^{\circ}.\label{eq:rectangular subdomain}
\end{equation}
\begin{figure}[H]
\includegraphics[scale=0.5]{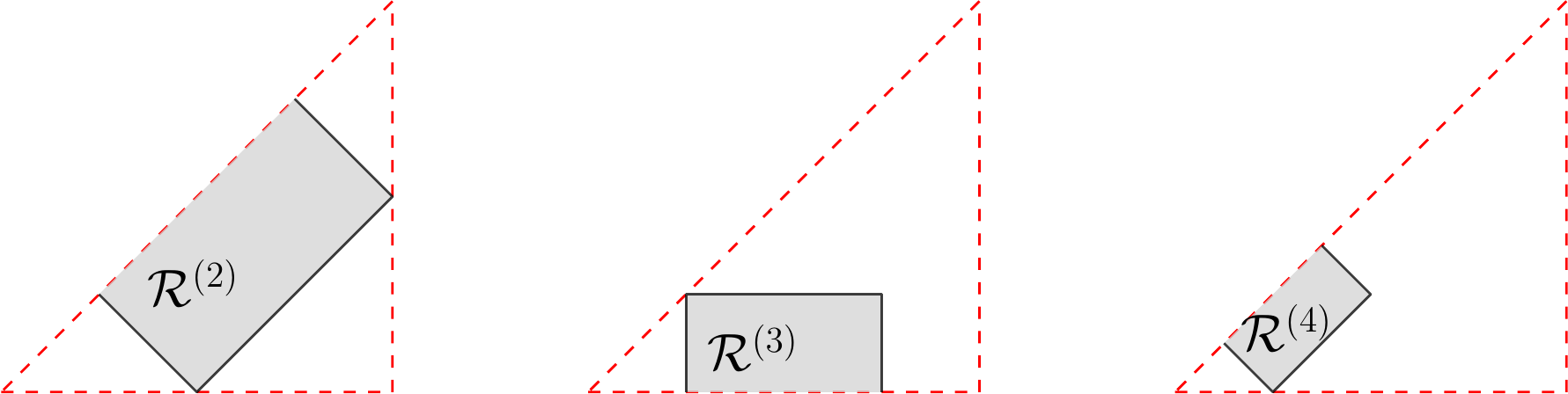}

\protect\caption{\label{fig:rectangle}The subdomain $\protect\recdom k$ as it appears
in the $k$-frame partition, $\protect\P^{(k)}$, for $k=2,3,4$.}
\end{figure}
\end{enumerate}
\begin{lem}
\label{lem:high unfoldings degenerate eigenvalues} Let $\left(m,n\right)\in\O\ ,\,n\neq0$.
Then the following holds:
\begin{enumerate}
\item Consider the eigenvalue problem \eqref{eq:subdomains eigenvalue problem}
on $\sdom$.\\
 Then $\lambda_{U_{\Q}\left(m,n\right)}\in\spec{\sdom}$ is a non-simple
eigenvalue.
\item Let $k>1$ and consider the eigenvalue problem \eqref{eq:subdomains eigenvalue problem}
on $\recdom k$. \\
Then $\lambda_{U_{\Q}^{k}\left(m,n\right)}\in\spec{\recdom k}$ is
a non-simple eigenvalue.
\end{enumerate}
\end{lem}
\begin{proof}
We start by giving explicit expressions for the eigenvalues and the
eigenfunctions of \eqref{eq:subdomains eigenvalue problem} on the
domains $\sdom$ and $\recdom k$. To do that, we choose the following
convenient parametrizations for the domains. First we consider $\sdom$,
which we write as 
\[
\sdom=\left(0,\frac{\pi}{2}\right)\times\left(0,\frac{\pi}{2}\right).
\]
 The boundary conditions induced by the 1-frame are expressed by 
\[
\hat{\varphi}\at_{\left\{ x=0\textrm{ or }y=0\right\} }\equiv0\ ;\ \ \frac{\partial\hat{\varphi}}{\partial n}\at_{\partial\sdom\backslash\left\{ x=0\textrm{ or }y=0\right\} }\equiv0.
\]
The eigenvalues are 
\begin{equation}
\hat{\lambda}_{p,q}=\left(2p+1\right)^{2}+\left(2q+1\right)^{2}\ \ ;\ \left(p,q\right)\in\N_{0}\times\N_{0}\label{eq:square eigenvalues}
\end{equation}
 and the orthogonal set of eigenfunctions is given by
\begin{equation}
\hat{\varphi}_{p,q}\left(x,y\right)=\sin\left(\left(2p+1\right)x\right)\sin\left(\left(2q+1\right)y\right)\ \ ;\ \left(p,q\right)\in\N_{0}\times\N_{0}.\label{eq:eq:square eigenfunctions}
\end{equation}
We proceed with $\recdom k$. The edges have ratio $1:2$ and the
longest one is of length $l=\frac{\pi}{2^{\frac{k-1}{2}}}$ , thus
we may write 
\[
\recdom k=\left(0,l\right)\times\left(0,\frac{l}{2}\right).
\]
The boundary conditions induced by the k-frame are expressed by
\[
\hat{\varphi}\at_{\partial\recdom k\backslash\left\{ y=\frac{l}{2}\right\} }\equiv0\ ;\ \ \frac{\partial\hat{\varphi}}{\partial n}\at_{\left\{ y=\frac{l}{2}\right\} }\equiv0.
\]
The eigenvalues are given to be
\begin{equation}
\hat{\lambda}_{p,q}=2^{k-1}\left(p^{2}+q^{2}\right)\ \ ;\ \left(p,q\right)\in\N\times\left[2\N_{0}+1\right]\label{eq:rec eigenvalues}
\end{equation}
 and the orthogonal set of eigenfunctions is given by
\begin{equation}
\hat{\varphi}_{p,q}\left(x,y\right)=\sin\left(\frac{\pi\cdot p\cdot x}{l}\right)\sin\left(\frac{\pi\cdot q\cdot y}{l}\right)\ \ ;\ \left(p,q\right)\in\N\times\left[2\N_{0}+1\right].\label{eq:rec eigenfunctions}
\end{equation}
We proceed to prove both parts of the lemma by pointing out on two
linearly independent eigenfunctions which correspond to the relevant
eigenvalue. Recall that we consider $\left(m,n\right)\in\O$, with
$n\neq0.$
\begin{enumerate}
\item Define 
\[
\left(p_{1},q_{1}\right)=\left(\frac{m+n-1}{2},\frac{m-n-1}{2}\right),
\]
\[
\left(p_{2},q_{2}\right)=\left(q_{1},p_{1}\right).
\]
As $m\neq n\,\,\left(\textrm{mod~2}\right)$ we get $\left(p_{1},q_{1}\right),\,\,\left(p_{2},q_{2}\right)\in\N_{0}\times\N_{0}$
and since $n\neq0$ we get $\left(p_{1},q_{1}\right)\neq\left(p_{2},q_{2}\right)$.
By \eqref{eq:eq:square eigenfunctions} we get that $\hat{\varphi}_{p_{1},q_{1}}$
and $\hat{\varphi}_{p_{2},q_{2}}$ are linearly independent and by
\eqref{eq:square eigenvalues} we obtain 
\[
2\left(m^{2}+n^{2}\right)=\lambda_{U_{\Q}\left(m,n\right)}=\hat{\lambda}_{p_{1},q_{1}}=\hat{\lambda}_{p_{2},q_{2}}.
\]
Thus the eigenvalue $\lambda_{U_{\Q}\left(m,n\right)}\in\spec{\sdom}$
is non-simple.
\item \noindent Define 
\[
\left(p_{1},q_{1}\right)=\left(m+n,m-n\right),
\]
\[
\left(p_{2},q_{2}\right)=\left(q_{1},p_{1}\right).
\]
As $m\neq n\,\,\left(\textrm{mod2}\right)$ we get $\left(p_{1},q_{1}\right),\,\,\left(p_{2},q_{2}\right)\in\N\times\left[2\N_{0}+1\right]$
and since $n\neq0$ we get $\left(p_{1},q_{1}\right)\neq\left(p_{2},q_{2}\right)$.
By \eqref{eq:rec eigenfunctions} we get that $\hat{\varphi}_{p_{1},q_{1}}$
and $\hat{\varphi}_{p_{2},q_{2}}$ are linearly independent and by
\eqref{eq:rec eigenvalues} we obtain 
\[
2^{k}\left(m^{2}+n^{2}\right)=\lambda_{U_{\Q}^{k}\left(m,n\right)}=\hat{\lambda}_{p_{1},q_{1}}=\hat{\lambda}_{p_{2},q_{2}}.
\]
Thus the eigenvalue $\lambda_{U_{\Q}^{k}\left(m,n\right)}\in\spec{\recdom k}$
is non-simple.
\end{enumerate}
\end{proof}
Proposition \ref{prop:rulling_all_courant_sharp_traingle},\eqref{enu:prop_almost all high unfoldings are non courant sharp}
follows immediately by combining Lemma \ref{lem:CS_implies_simplicity},\eqref{enu:lem_CS_implies_simplicity_1}
with Lemma \ref{lem:high unfoldings degenerate eigenvalues}.

\subsection{Proving Proposition \ref{prop:rulling_all_courant_sharp_traingle},\eqref{enu:prop_The-eigenvalues-of U(m,0) not courant sharp}.\protect \\
}

By Lemma \ref{lem:CS_implies_simplicity},\eqref{enu:lem_CS_implies_simplicity_2}
it follows that we only need to rule out the Courant-sharpness of
the simple eigenvalues of $\Lambda^{(3)}$.
\begin{proof}[Proof of Proposition \ref{prop:rulling_all_courant_sharp_traingle},\eqref{enu:prop_The-eigenvalues-of U(m,0) not courant sharp}]
 In order to show that the simple eigenvalues of $\Lambda^{(3)}$
are not Courant-sharp, we find in the following a subset $\qtrngle{\lambda_{m,n}}\varsubsetneq\Q\left(\lambda_{m,n}\right)\cup\left\{ \left(m,n\right)\right\} $
such that $\left|\qtrngle{\lambda_{m,n}}\right|=\nu\left(\varphi_{m,n}\right)$.
This will rule out Courant-sharpness of a simple eigenvalue, since
then 
\[
N\left(\lambda_{m,n}\right)=\left|\Q\left(\lambda_{m,n}\right)\cup\left\{ \left(m,n\right)\right\} \right|>\left|\qtrngle{\lambda_{m,n}}\right|=\nu\left(\varphi_{m,n}\right).
\]
Before proceeding, we rewrite $\Lambda^{(3)}$ as an expression that
is more adjusted to the following arguments
\[
\Lambda^{(3)}=\left\{ \lambda_{2m,0}\right\} _{m=3}^{\infty}\cup\left\{ \lambda_{m,m}\right\} _{m=3}^{\infty}.
\]
First, we treat the case of a simple eigenvalue $\lambda_{m,m}$ for
$m\geq3.$ Its nodal set is 
\[
\varphi_{m,m}^{-1}\left\{ 0\right\} =\left\{ \left(x,y\right)\in\trngle\at\cos\left(mx\right)\cos\left(my\right)=0\right\} .
\]
Thus we deduce that there are $m$ nodal lines inside $\trngle$ parallel
to $x$-axis and $m$ that are parallel to the $y$-axis (Figure \ref{fig:T_3_3 and varphi_3_3},(a)).
The number of nodal domains is therefore
\begin{equation}
\nu\left(\varphi_{m,m}\right)=\sum_{i=0}^{m}\left(i+1\right).\label{eq:nodal domains of k,k}
\end{equation}
\begin{figure}[H]
\begin{centering}
\label{fig:nodal_set_3_3_and_6_0}%
\begin{minipage}[t]{0.4\columnwidth}%
\begin{center}
(a)~~~~~~\includegraphics[scale=0.4]{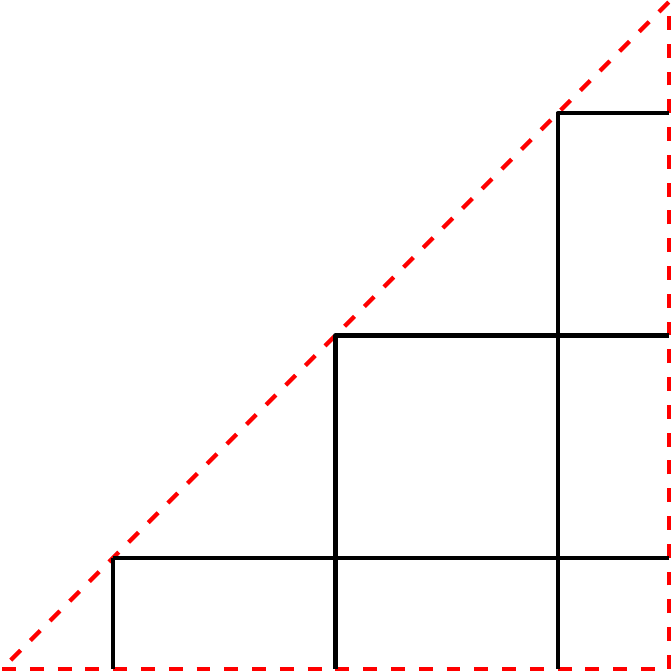}
\par\end{center}%
\end{minipage}%
\begin{minipage}[t]{0.4\columnwidth}%
(b)~~~~~~\includegraphics{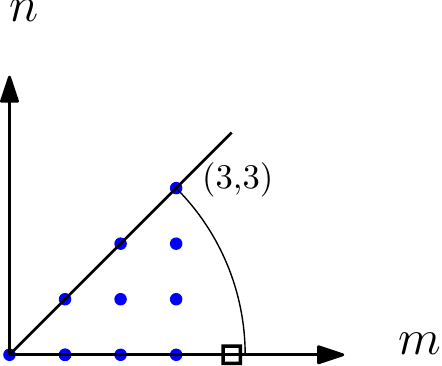}%
\end{minipage}
\par\end{centering}

\protect\caption{\label{fig:T_3_3 and varphi_3_3}(a) The nodal set of $\varphi_{3,3}$\quad{}(b)
The blue points correspond to $\protect\qtrngle{\lambda_{3,3}}$ and
the empty square corresponds to the point $\left(4,0\right)$.}
\end{figure}

Denote
\[
\qtrngle{\lambda_{m,m}}:=\left\{ \left(i,j\right)~\at\,0\leq j\leq i\leq m\right\} ,
\]
and observe that 
\begin{equation}
\left|\qtrngle{\lambda_{m,m}}\right|=\sum_{i=0}^{m}\left(i+1\right)=\nu\left(\varphi_{m,m}\right).\label{eq:arithmetic progression sum a,a}
\end{equation}

A simple calculation shows that (See Figure \ref{fig:T_3_3 and varphi_3_3},(b))
\[
\qtrngle{\lambda_{m,m}}\subseteq\Q\left(\lambda_{m,m}\right)\cup\left\{ \left(m,m\right)\right\} .
\]
In addition, for $m\geq3$ we have $\left(m+1,0\right)\in\Q\left(\lambda_{m,m}\right)\backslash\qtrngle{\lambda_{m,m}}$,
since 
\[
\left\Vert \left(m+1,0\right)\right\Vert ^{2}=m^{2}+2m+1<2m^{2}=\lambda_{m,m}.
\]
Thus we showed
\[
\qtrngle{\lambda_{m,m}}\subsetneq\Q\left(\lambda_{m,m}\right)\cup\left\{ \left(m,m\right)\right\} .
\]
Next, we treat the case of a simple eigenvalue $\lambda_{2m,0}\in\spec{\trngle}$
for $m\geq3$. This eigenvalue is the unfolding of $\lambda_{m,m}$
which we treated above (and therefore their multiplicity is equal).
Its nodal set is therefore determined easily (Figure \ref{fig:T_4_0}
,(a)) and the nodal count is given by 
\begin{equation}
\nu\left(\varphi_{2m,0}\right)=2\nu\left(\varphi_{m,m}\right)-\left(m+1\right)=\sum_{i=0}^{m}\left(2i+1\right).\label{eq:nodal domains of k,0}
\end{equation}
Denote
\[
\qtrngle{\lambda_{2m,0}}=\left\{ \left(m+j,m-i\right)\,\at\,0\leq i\leq m\ ,\,-i\leq j\leq i\right\} .
\]

\begin{figure}[H]
\begin{centering}
\label{fig:nodal_set_3_3_and_6_0-1}%
\begin{minipage}[t]{0.4\columnwidth}%
\begin{center}
(a)~~~~~~\includegraphics[scale=0.4]{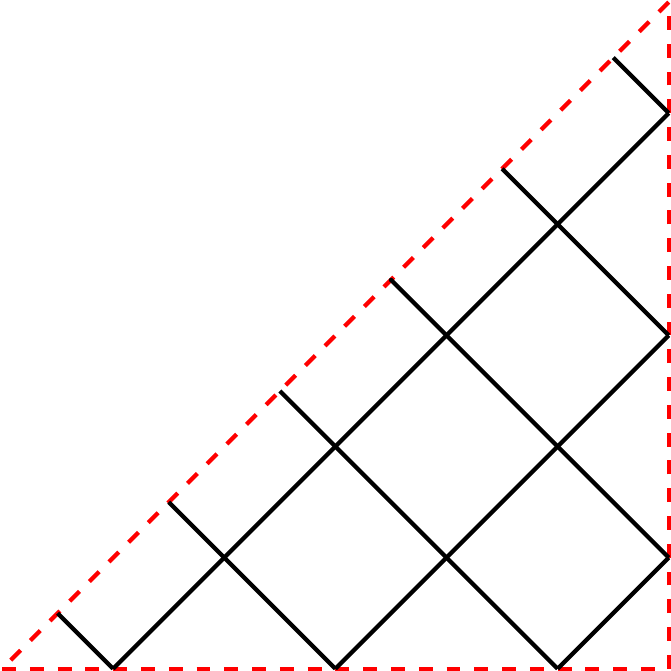}
\par\end{center}%
\end{minipage}%
\begin{minipage}[t]{0.4\columnwidth}%
(b)~~~~~~\includegraphics[scale=0.9]{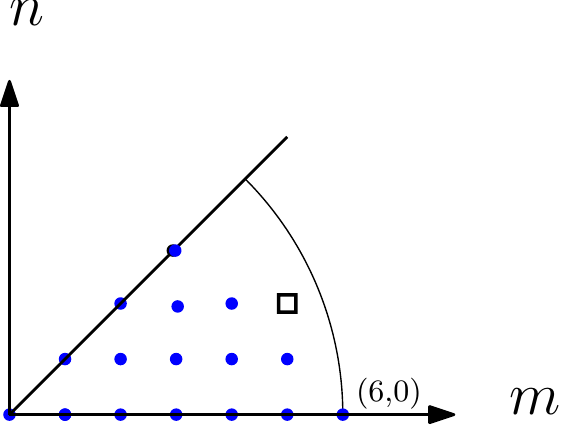}%
\end{minipage}
\par\end{centering}

\protect\caption{\label{fig:T_4_0}(a) The nodal set of $\varphi_{6,0}$\quad{}(b)
The blue points correspond to $\protect\qtrngle{\lambda_{6,0}}$ and
the empty square corresponds to the point $\left(5,2\right)$.}
\end{figure}
Observe that 
\begin{equation}
\left|\qtrngle{\lambda_{2m,0}}\right|=\sum_{i=0}^{m}\sum_{j=-i}^{i}1=\sum_{i=0}^{m}\left(2i+1\right)=\nu\left(\varphi_{2m,0}\right),\label{eq:arithmetic progression sum a,0}
\end{equation}
and a simple calculation shows that (See figure \ref{fig:T_4_0},
(b))
\[
\qtrngle{\lambda_{2m,0}}\subseteq\Q\left(\lambda_{2m,0}\right)\cup\left\{ \left(2m,0\right)\right\} .
\]
Observe that for $m\geq3$ we have $\left(2m-1,2\right)\in\Q\left(\lambda_{2m,0}\right)\backslash\qtrngle{\lambda_{2m,0}}$
since 
\[
\left\Vert \left(2m-1,2\right)\right\Vert ^{2}=4m^{2}-4m+5<4m^{2}=\lambda_{2m,0}.
\]
Thus we showed
\[
\qtrngle{\lambda_{2m,0}}\varsubsetneq\Q\left(\lambda_{2m,0}\right)\cup\left\{ \left(2m,0\right)\right\} .
\]

\end{proof}

\subsection{Concluding the proof of Theorem \ref{thm:CourantSharp_triangle}.\protect \\
}

Finally, Theorem \ref{thm:CourantSharp_triangle} is proved once we
show that the eigenvalues we have not ruled out are indeed Courant-sharp.
\begin{lem}
\label{lem:cournat sharpness of the eigenvals}The eigenvalues of
$\mathcal{C}=\{0\}\cup\left\{ \lambda_{U_{Q}^{k}\left(1,0\right)}\right\} _{k=0}^{3}$
are Courant-sharp\end{lem}
\begin{proof}
By Courant's bound and orthogonality, the first two eigenvalues, $\lambda_{0,0}=0$
and $\lambda_{1,0}=1$ are Courant-sharp. Next, note that the eigenvalue
$\lambda_{1,0}$ is simple and therefore all of the eigenvalues in
$\left\{ \lambda_{U_{Q}^{k}\left(1,0\right)}\right\} _{k=1}^{3}$
are simple as well. It is now straightforward to find the number of
nodal domains of the eigenfunctions in the set $\left\{ \varphi_{U_{Q}^{k}\left(1,0\right)}\right\} _{k=1}^{3}$
(see for example \eqref{eq:nodal domains of k,k},\eqref{eq:nodal domains of k,0})
and verify that those three eigenvalues are Courant-sharp as well.
\end{proof}
We end by noting that the nodal sets of the non-constant Courant-sharp
eigenfunctions are exactly the first four $k$-frames (see Figure
\ref{fig:four_unfoldings-1}).

\section{Proof of Theorem \ref{thm:CourantSharp_boxes} \label{sec:proof_for_boxes}}

We start by developing the eigenfunction folding structure of an $n$-dimensional
box, $\nbox$, whose edge length ratio is given by $\frac{l_{j}}{l_{j+1}}=\ohm:=2^{\nicefrac{1}{n}}$
($1\leq j\leq n-1$). For convenience we choose a scaling according
to which $l_{1}=\pi$. We start by following the construction from
Section \ref{sec:folding_structure_triangle} and present the folding
structure of the $\nbox$ eigenfunctions.

The set of quantum numbers in this case is 
\begin{equation}
\Q:=\left\{ \mvec\in\mathbb{N}_{0}^{n}\right\} .\label{eq:quantum_numbers_box}
\end{equation}
The orthogonal basis of eigenfunctions is 
\begin{equation}
\varphi_{\mvec}(\xvec)=\prod_{j=1}^{n}\cos(\ohm^{j-1}m_{j}x_{j})\ ;\ \mvec\in\Q\label{eq:eigenfunction_box}
\end{equation}
and the corresponding eigenvalues are
\begin{equation}
\lambda_{\mvec}=\sum_{j=1}^{n}\left(\ohm^{j-1}m_{j}\right)^{2}~~;\ \mvec\in\Q.\label{eq:eigenvalue_box}
\end{equation}
We use the notations 
\[
\spec{\nbox}:=\left\{ \lambda_{\mvec}\,\at\,\mvec\in\Q\right\} ,
\]

\[
\Q(\lambda):=\left\{ \mvec\in\Q\,\at\,\lambda_{\mvec}<\lambda\right\} 
\]

and have as before that
\[
\ndown\left(\lambda\right)=\left|\Q(\lambda)\right|.
\]

The box $\nbox$ is symmetric with respect to the following hyperplane
\begin{equation}
\line=\left\{ \left.\xvec\in\nbox\right|x_{1}=\frac{\pi}{2}\right\} ,\label{eq:hyperplabe_of_box}
\end{equation}

and the reflection transformation is
\begin{equation}
R\left(\xvec\right)=\left(\pi-x_{1},~x_{2},\ldots,x_{n}\right).\label{eq:reflection_box}
\end{equation}

As opposed to the case of the triangle, the eigenvalues of $\nbox$
are not integers (with the exception of the case $n=2$, where they
are), but rather belong to $\Z\left[\ohm^{2}\right]$, a finite ring
extension of $\Z$. We consider $\Z\left[\ohm^{2}\right]$ as a free
module with the following basis 
\begin{equation}
\bas=\begin{cases}
\left\{ \ohm^{j}\right\} _{j=0}^{n-1} & ~~~n\textrm{~is odd}\\
\\
\left\{ \ohm^{2j}\right\} _{j=0}^{\nicefrac{n}{2}-1} & ~~~n\textrm{~is even}
\end{cases}.\label{eq:generating_set}
\end{equation}
Furthermore, in the unique representation of $\lambda\in\spec{\nbox}$
as a linear combination of this basis, the coefficients are taken
from $\N_{0}$. This is used to define the parity of an eigenvalue.
\begin{defn}
\label{def:parity_evalue_box} Denoting by $\pr{\lambda}$ the coefficient
multiplying $\ohm^{0}=1$ when spanning $\lambda\in\spec{\nbox}$
by $\bas$, we call $\lambda$ an \emph{odd }(\emph{even})\emph{ eigenvalue}
if $\pr{\lambda}$ is odd (even).
\end{defn}
Hence, we adopt here the dichotomy to even and odd eigenvalues, similarly
to the one we had in Section \ref{sec:folding_structure_triangle}.
The parity of an eigenvalue dictates the parity of all of its eigenfunctions
with respect to the reflection across $\line$, which is proved in
the following (analogously to Lemma \ref{lem:eigenfunction_symmetry}).
\begin{lem}
\label{lem:eigenfunction_symmetry_box}\emph{ }Let $\lambda\in\sigma\left(\nbox\right)$,
then any eigenfunction of $\lambda$ is odd (even) with respect to
$\line$\emph{ if and only if }$\lambda$ is an odd (even) eigenvalue.\end{lem}
\begin{proof}
Let $\mvec\in\Q$ such that $\varphi_{\mvec}$ is an eigenfunction
corresponding to $\lambda$. Writing $\lambda$ as $\sum_{j=1}^{n}\left(\ohm^{j-1}m_{j}\right)^{2}$
and using that $\bas$ is a basis we have that 
\begin{equation}
\pr{\lambda}=\begin{cases}
m_{1}^{2} & ~~\textrm{n is odd}\\
m_{1}^{2}+2m_{\nicefrac{n}{2}+1}^{2} & ~~\textrm{n is even}
\end{cases},\label{eq:parity_of_lambda}
\end{equation}
and in both cases the parity of $\lambda$ equals the parity of $m_{1}$.
From the explicit expression of the eigenfunction, \eqref{eq:eigenfunction_box},
we see that $\varphi_{\mvec}$ is odd (even) with respect to $L$
if and only if $m_{1}$ is odd (even). If $\lambda$ is a multiple
eigenvalue and it is odd (even), the argument above gives that the
basis, $\left\{ \varphi_{\mvec}~\at~\lambda_{\mvec}=\lambda\right\} $,
of its eigenspace consists of odd (even) eigenfunctions and therefore
so is any eigenfunction of $\lambda$.
\end{proof}
\noindent As in Section \ref{sec:folding_structure_triangle}, Lemma
\ref{lem:eigenfunction_symmetry_box} motivates the following definition
(compare with Definition \ref{def:odd_even_spectrum}).
\begin{defn}
\label{def:odd_even_spectrum_box} We define the subsets of $\Q$
that correspond to the odd and even eigenvalues

\begin{equation}
\begin{array}{c}
\O\coloneqq\left\{ \vec{m}\in\Q\,|\,\,m_{1}=1\,(mod~2)\right\} ,\\
\\
\E\coloneqq\left\{ \vec{m}\in\Q\,|\,\,m_{1}=0\,(mod~2)\right\} 
\end{array}\label{eq:odd_even_box_qn}
\end{equation}
and we denote the corresponding sets of eigenvalues by
\[
\ospec{\B^{(n)}}\coloneqq\left\{ \lambda_{\vec{m}}\at\vec{m}\in\O\right\} ,
\]
\[
\espec{\B^{(n)}}\coloneqq\left\{ \lambda_{\vec{m}}\at\vec{m}\in\E\right\} .
\]

\end{defn}
\noindent Denote 
\[
\hfbox=\left\{ \xvec\in\trngle:\,x_{1}\leq\frac{\pi}{2}\right\} .
\]

\noindent Observe that $\line$ partitions $\nbox$ into the two isometric
boxes $\hfbox$ and $\left(\nbox\backslash\hfbox\right)\cup\line$,
each is a scaled version of $\nbox$ by a factor $\ohm$. Namely,
\begin{equation}
\left(l_{1},\ldots,l_{n-1},l_{n}\right)=\ohm\cdot\left(l_{2},~\ldots,l_{n},~\frac{l_{1}}{2}\right),\label{eq:measures_of_box_and_half_box}
\end{equation}
where the left-hand side gives the edge lengths of $\nbox$ and the
right-hand side the edge lengths of $\hfbox$ (see Figure \ref{fig:box and half box}).
\begin{rem*}
Equation \eqref{eq:measures_of_box_and_half_box} may be perceived
as a generalization of the A-series ($A_{3},$ $A_{4}$, etc.) paper
sizes to higher dimensions.
\end{rem*}
\noindent 
\begin{figure}
\includegraphics[scale=0.75]{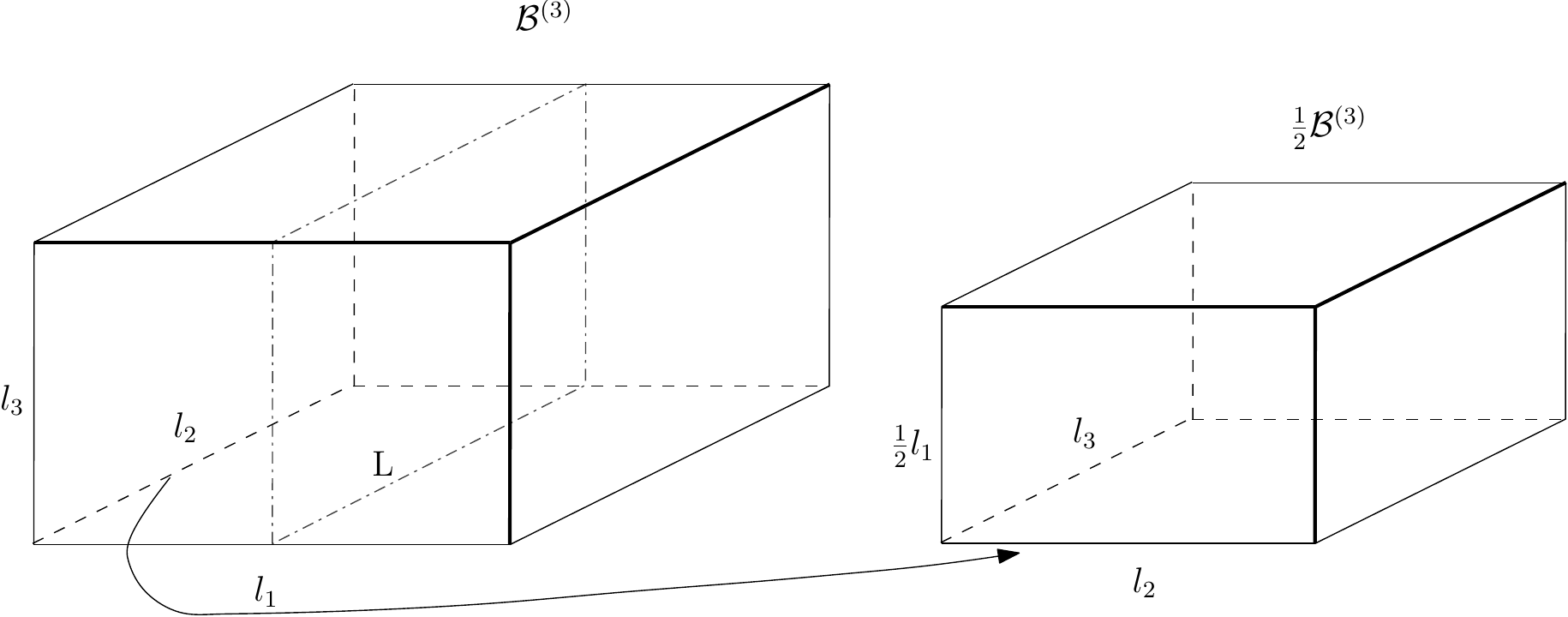}

\protect\caption{\label{fig:box and half box}Illustration of $\protect\B^{(3)}$ decomposed
into two similar boxes, one of which is $\frac{1}{2}\protect\B^{(3)}$
(up to rotation).}
\end{figure}
This similarity reveals the folding structure of the $\nbox$ eigenfunctions.
Indeed the following two definitions and lemma are analogous to Definitions
\ref{def:fold_unfold_coords}, \ref{def:fold_unfold_efunc} and Lemma
\ref{lem:fold_unfold_efunc} of the triangle case.
\begin{defn}
\label{def:fold_unfold_coords_box} The \emph{coordinate folding transformation}
is

\begin{align}
\begin{array}{c}
F:\hfbox\to\nbox\\
\\
F\left(x_{1},x_{2},\ldots,x_{n}\right)\coloneqq\ohm\cdot\left(x_{2},~x_{3},\ldots,x_{n},~x_{1}\right)
\end{array}\label{eq:fold_coordinates_box}
\end{align}

and the \emph{coordinate unfolding transformation} $U$ is the inverse
of $F$ and is expressed by 
\begin{equation}
\begin{array}{c}
U:\nbox\to\hfbox\\
\\
U\left(x_{1},x_{2},\ldots,x_{n}\right)=\ohm^{-1}\cdot\left(x_{n},~x_{1},\ldots,x_{n-1}\right).
\end{array}\label{eq:unfold_coords_box}
\end{equation}

\end{defn}
~
\begin{defn}
~\label{def::fold_unfold_efunc_box} Let $\varphi$ be an eigenfunction
corresponding to the eigenvalue $\lambda\in\sigma\left(\nbox\right)$.
\begin{enumerate}
\item Assume $\lambda$ is even. Then the \emph{folded function} $\fold\varphi$
is defined by 
\begin{equation}
\fold\varphi\left(\xvec\right)=\varphi\circ U\left(\xvec\right)\,,\,\xvec\in\nbox.\label{eq:fold_fcn_box}
\end{equation}

\item The \emph{unfolded function},$\unfold\varphi$, is
\begin{equation}
\unfold\varphi\left(\xvec\right)=\begin{cases}
\varphi\circ F\left(\xvec\right) & \xvec\in\hfbox\\
\left(\varphi\circ F\right)\circ R\left(\xvec\right) & \xvec\in\nbox\backslash\hfbox
\end{cases}.\label{eq:unfold_fcn_box}
\end{equation}

\end{enumerate}
\end{defn}
\begin{lem}
\label{lem:fold_unfold_efunc_box} Let $\varphi=\underset{\mvec;~\lambda_{\mvec}=\lambda}{\sum}\alpha_{\mvec}\cdot\varphi_{\mvec}$
be an eigenfunction corresponding to the eigenvalue $\lambda\in\sigma\left(\nbox\right)$. 
\begin{enumerate}
\item If $\lambda$ is even, then the folded function is $\fold\varphi=\underset{\mvec;~\lambda_{\mvec}=\lambda}{\sum}\alpha_{\mvec}\cdot\varphi_{F_{\Q}\left(\mvec\right)}$
and corresponds to the eigenvalue $\lambda_{F_{\Q}\left(\mvec\right)}=\ohm^{-2}\lambda_{\mvec}$,
with \emph{
\begin{equation}
\begin{array}{c}
F_{\Q}\left(\mvec\right):=\left(m_{2},~m_{3},\ldots,m_{n},~\frac{m_{1}}{2}\right)\end{array}.\label{eq:fold_quant_numbers_box}
\end{equation}
}
\item The unfolded function is $\unfold\varphi=\underset{\mvec;~\lambda_{\mvec}=\lambda}{\sum}\alpha_{\mvec}\cdot\varphi_{U_{\Q}\left(\mvec\right)}$
and corresponds to the eigenvalue $\lambda_{U_{\Q}\left(\mvec\right)}=\ohm^{2}\lambda_{\mvec}$,
with
\begin{equation}
\begin{array}{c}
U_{\Q}\left(\mvec\right):=\left(2m_{n},~m_{1},~m_{2},\ldots,m_{n-1}\right).\end{array}\label{eq:unfold_quant_numbers_box}
\end{equation}

\end{enumerate}
\end{lem}
~
\begin{proof}
~
\begin{enumerate}
\item Let $\lambda$ be an even eigenvalue of $\nbox$. Let $\mvec\in\Q$
be such that $\lambda_{\mvec}=\lambda$. As $\pr{\lambda}$ is even
we conclude that $m_{1}$ is even as well (see \eqref{eq:parity_of_lambda})
and therefore, $F_{\Q}\left(\mvec\right)\in\Q$ so that $\varphi_{F_{\Q}\left(\mvec\right)}$
is well defined and it is an eigenfunction of $\nbox$. Combining
the form of the eigenfunction $\varphi_{\mvec}$, \eqref{eq:eigenfunction_box},
with the definition of its folding, \eqref{eq:fold_fcn_box}, we get
\begin{align*}
\fold\varphi_{\mvec}\left(\xvec\right) & =\varphi_{\mvec}\circ U\left(\xvec\right)\\
 & =\cos\left(m_{1}\ohm^{-1}x_{n}\right)\prod_{j=2}^{n}\cos\left(\ohm^{j-1}m_{j}\ohm^{-1}x_{j-1}\right)\\
 & =\cos\left(\ohm^{n-1}\frac{1}{2}m_{1}x_{n}\right)\prod_{j=1}^{n-1}\cos\left(\ohm^{j-1}m_{j+1}x_{j}\right)\\
 & =\varphi_{F_{\Q}\left(\mvec\right)}\left(\xvec\right).
\end{align*}
If $\lambda$ is a multiple eigenvalue, the calculation above is valid
for any eigenfunction of the form $\varphi_{\mvec}$ and by linearity
it extends to $\fold\varphi=\underset{\mvec;~\lambda_{\mvec}=\lambda}{\sum}\alpha_{\mvec}\cdot\varphi_{F_{\Q}\left(\mvec\right)}$.
Calculating the eigenvalue corresponding to $\varphi_{F_{\Q}\left(\mvec\right)}$
we get 
\begin{equation}
\lambda_{F_{\Q}\left(\mvec\right)}=\ohm^{-2}m_{1}^{2}+\sum_{j=1}^{n-1}\left(\ohm^{j-1}m_{j+1}\right)^{2}=\ohm^{-2}\sum_{j=1}^{n}\left(\ohm^{j-1}m_{j}\right)^{2}=\ohm^{-2}\lambda_{\mvec}.\label{eq:folding_evalue_box}
\end{equation}

\item \noindent Let $\lambda\in\spec{\nbox}$ and let $\mvec\in\Q$ such
that $\lambda_{\mvec}=\lambda$. Let $\xvec\in\hfbox$.
\begin{align*}
\unfold\varphi_{\mvec}\left(\xvec\right) & =\varphi_{\mvec}\circ F\left(\xvec\right)\\
 & =\cos\left(\ohm^{n-1}m_{n}\ohm x_{1}\right)\prod_{j=1}^{n-1}\cos\left(\ohm^{j-1}m_{j}\ohm x_{j+1}\right)\\
 & =\cos\left(2m_{n}x_{1}\right)\prod_{j=2}^{n}\cos\left(\ohm^{j-1}m_{j-1}x_{j}\right)\\
 & =\varphi_{U_{\Q}\left(\mvec\right)}\left(\xvec\right).
\end{align*}
For $\xvec\in\nbox\backslash\hfbox$ we have
\[
\unfold\varphi_{\mvec}\left(\xvec\right)\underbrace{=}_{\text{\eqref{eq:unfold_fcn_box}}}\unfold\varphi_{\mvec}\left(R\left(\xvec\right)\right)\underbrace{=}_{\text{\ensuremath{R\left(\xvec\right)\in\hfbox}}}\varphi_{U_{\Q}\left(\mvec\right)}\left(R\left(\xvec\right)\right)\underbrace{=}_{\text{Lemma \ref{lem:eigenfunction_symmetry_box}}}\varphi_{U_{\Q}\left(\mvec\right)}\left(\xvec\right).
\]
 Just as in the first part of the proof, we may use linearity to extend
the relation above to the whole eigenspace of $\lambda$. In addition,
it is easily verified that $\lambda_{U_{\Q}\left(\mvec\right)}=\ohm^{2}\lambda_{\mvec}.$
\end{enumerate}
\end{proof}
The last lemma allows to show that the eigenvalues inherit the folding
structure. This is shown in the following, which is analogous to Corollary
\ref{cor:eiganvalue_is_odd_times_power_of_two}.
\begin{cor}
\begin{singlespace}
~\label{cor:eiganvalue_is_odd_times_power_of_gamma}\end{singlespace}

\begin{enumerate}
\item \label{enu:cor_eiganvalue_is_odd_times_power_of_gamma_1}Let $0\neq\lambda\in\spec{\nbox}$.
Then there exist unique $\oval\in\ospec{\nbox}$ and $k\in\N_{0}$
such that $\lambda=\ohm^{2k}\oval$. Furthermore, $\mult\left(\lambda\right)=\mult\left(\oval\right)$.
\item \label{enu:cor_eiganvalue_is_odd_times_power_of_gamma_2}Let $\oval\in\ospec{\nbox}$
and $k\in\N_{0}$. Then $\ohm^{2k}\oval\in\spec{\nbox}$.
\end{enumerate}
\end{cor}
\begin{proof}
We start by observing that the second claim may be proven similarly
to the second claim of Corollary \ref{cor:eiganvalue_is_odd_times_power_of_two}
- start from any eigenfunction of $\oval$ and by unfolding it $k$
times get an eigenfunction whose eigenvalue is $\ohm^{2k}\oval$. 

Next, we prove the first claim and start by proving the uniqueness
of the representation $\lambda=\ohm^{2k}\oval$. Assume by contradiction
that $\ohm^{2k_{1}}\oval_{1}=\ohm^{2k_{2}}\oval_{2}$, with different
$k_{1},k_{2}\in\N_{0}$ and $\oval_{1},\oval_{2}\in\ospec{\nbox}$.
Without loss of generality, $k_{1}>k_{2}$ and hence $\oval_{2}=\ohm^{2\left(k_{1}-k_{2}\right)}\oval_{1}$.
Pick an eigenfunction of $\oval_{1}$ and unfold it $k_{1}-k_{2}$
times to get an eigenfunction of the eigenvalue $\oval_{2}$. By \eqref{eq:unfold_fcn_box}
this unfolded eigenfunction is even and by Lemma \ref{lem:eigenfunction_symmetry_box}
we deduce that its eigenvalue, $\oval_{2}$ is also even and arrive
at a contradiction.

It remains to show the existence of $k\in\N_{0},~\oval\in\ospec{\nbox}$
such that $\lambda=\ohm^{2k}\oval$. There exists some $\mvec\in\Q$
such that $\lambda=\lambda_{\mvec}$. If $\lambda\in\ospec{\nbox}$
then the statement holds with $k=0$. Otherwise, by Lemma \ref{lem:fold_unfold_efunc_box},
we get that $\lambda_{F_{\Q}\left(\mvec\right)}=\ohm^{-2}\lambda_{\mvec}$
is an eigenvalue. We keep applying $F_{\Q}$ to $\mvec$ until we
get that $\lambda_{F_{\Q}^{k}\left(\mvec\right)}=\ohm^{-2k}\lambda_{\mvec}$
is an odd eigenvalue. Once we get that, the lemma is proved and it
only remains to show that this process terminates after a finite ($k$)
number of steps. 

In order to see this we may present the $\mvec$ entries as $m_{j}=p_{j}2^{k_{j}}$,
with $k_{j}$ being the largest possible (and formally set $p_{j}=0,~k_{j}=\infty$
if $m_{j}=0$). The subsequent applications of $F_{\Q}$ cyclically
shift the vector and divide the first entry by two (see \eqref{eq:fold_quant_numbers_box}).
Eventually, one of the entries would be odd and the process stops
(unless $\lambda=0$). 

Finally, the equality of multiplicities of $\lambda$ and $\oval$
arises as $\fold^{k}$ is a linear isomorphism (its inverse is $\unfold^{k}$)
from the eigenspace of $\lambda$ to the eigenspace of $\oval$.
\end{proof}
Defining the $k$-frame exactly as in \eqref{eq:k_frame} allows to
prove an analogue of Proposition \ref{prop:k_frame_vanishing}, namely
that for $\oval\in\ospec{\nbox}$, its $k$-unfolded eigenvalue, $\lambda=\ohm^{2k}\oval$
vanishes on the $k$-frame. This in turn shows that Lemmata \ref{lem:degeneracy of subdomain}
and \ref{lem:CS_implies_simplicity} are valid for the high-dimensional
boxes as well (with the $k$-frame partition defined just as in Definition
\ref{def:k_frame_partition}). All we need to use now is Lemma \ref{lem:CS_implies_simplicity},(2),
according to which multiple eigenvalues\footnote{See Appendix \ref{sec:appendix_multiplicity_boxes}, where we discuss
the possible eigenvalue multiplicities of the problem.} cannot be Courant-sharp. Alternatively, we may use Lemma \ref{lem:appendix_nodal_deficiency}
which is a generalization of Lemma \ref{lem:CS_implies_simplicity},(2).

We are left to check the Courant-sharpness of simple eigenvalues.
Since the nodal set of the basis eigenfunctions $\varphi_{\mvec}$,
is determined by
\[
\prod_{j=1}^{n}\cos(\ohm^{j-1}m_{j}x_{j})=0\,;\,\,0\leq x_{j}\leq\pi/\ohm^{j-1},
\]
it is straightforward to deduce that 
\[
\nu(\varphi_{\mvec})=\prod_{j=1}^{n}(m_{j}+1).
\]
This is compared with the spectral position in the next proposition
which rules out the Courant-sharpness of all eigenvalues not appearing
in Theorem \ref{thm:CourantSharp_boxes}.

\begin{figure}
\includegraphics[scale=0.6]{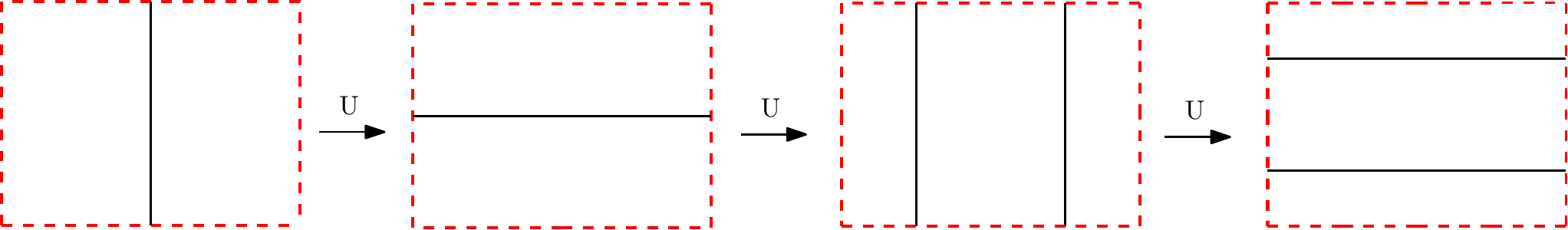}

\protect\caption{The first four k-frames of $\protect\B^{(2)}$\label{fig:four_unfoldings_rectangle}}
\end{figure}

\begin{prop}
\label{prop:non_CS_box} ~
\begin{enumerate}
\item For $n\geq3$ and $N\geq2$, $\lambda_{N}$ is not a Courant-sharp
eigenvalue of $\nbox$.
\item For $n=2$ and $N\notin\left\{ 1,2,4,6\right\} $, $\lambda_{N}$
is not a Courant-sharp eigenvalue for $\B^{\left(2\right)}$.
\end{enumerate}
\end{prop}
\begin{proof}
By the analogue of Lemma \ref{lem:CS_implies_simplicity},\eqref{enu:lem_CS_implies_simplicity_2}
(see discussion before this proposition), we only need to rule out
the Courant-sharpness of simple eigenvalues. Let $\lambda_{\vec{m}}\in\spec{\nbox}$
be a simple eigenvalue.

\noindent Let $\qbox{\lambda}=\Q\cap\left\{ \left(\left.\left(\tilde{m}_{1},\ldots,\tilde{m}_{n}\right)\right|\tilde{m}_{j}\leq m_{j}\,,\,\forall j\right)\right\} $.
Note that $\qbox{\lambda}$ contains all $\Q$-points contained in
an $n$-dimensional box and $\Q\left(\lambda\right)$ forms all the
$\Q$-points contained within an $n$-dimensional ellipsoid (see Figure
\ref{fig:B_Q_and_Q_lambda} for the $n=2$ case).

\noindent In the sequel we show $\qbox{\lambda}\subsetneq\Q\left(\lambda\right)\cup\left(m_{1},\ldots,m_{n}\right)$
which rules out Courant-sharpness since it gives 
\[
N\left(\lambda\right)=\left|\Q\left(\lambda\right)\cup\left(m_{1},\ldots,m_{n}\right)\right|>\left|\qbox{\lambda_{\vec{m}}}\right|=\nu\left(\varphi_{\vec{m}}\right).
\]
 It is easily seen that 
\[
\qbox{\lambda}\subseteq\Q\left(\lambda\right)\cup\left(m_{1},\ldots,m_{n}\right),
\]
and to show that 
\[
\qbox{\lambda}\subsetneq\Q\left(\lambda\right)\cup\left(m_{1},\ldots,m_{n}\right),
\]
 we point out $\mvec'\in\Q$ such that $\mvec'\in\Q\left(\lambda\right)\backslash\qbox{\lambda}$.
Note that this proof technique resembles the one which is used in
the proof of Proposition \ref{prop:rulling_all_courant_sharp_traingle},\eqref{enu:prop_The-eigenvalues-of U(m,0) not courant sharp}
and the set $\qbox{\lambda}$ plays the same role as the set $\qtrngle{\lambda}$
there. 
\begin{figure}
\noindent \includegraphics[scale=0.8]{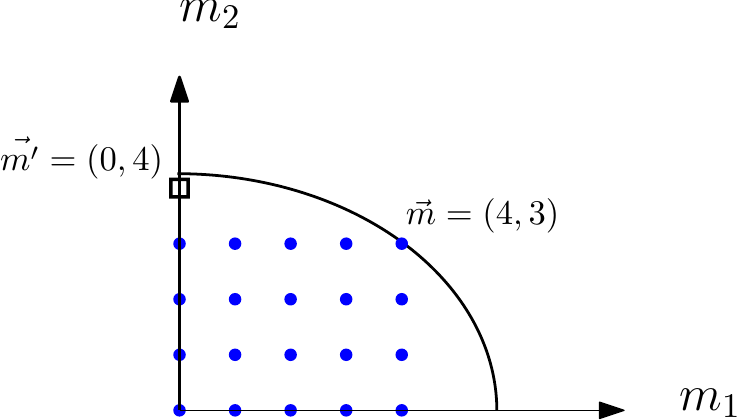}

\protect\caption{Illustration of $\protect\B_{\protect\Q}\left(\lambda_{4,3}\right)$}
\label{fig:B_Q_and_Q_lambda}
\end{figure}

Start by assuming that $\mvec$ is such that there exists $k$ for
which $m_{k}<m_{k+1}$. Choosing 
\begin{equation}
\mvec'=(0,..,0,\underbrace{m_{k}+1,}_{k-th~position}0,..,0),\label{eq:choosing_m_tag}
\end{equation}
satisfies
\[
\lambda_{\vec{m}'}<\lambda_{(0,..,0,\underbrace{m_{k+1},}_{k+1-th~position}0,..,0)}\leq\lambda_{\vec{m}}.
\]
 We may therefore proceed by assuming that the entries of $\mvec$
form a non-increasing ordered set. 

We distinguish the non-increasing sequences by setting 

\[
I:=\min\left\{ j~\at\,m_{j}=\min_{1\leq q\leq n}m_{q}\right\} ,
\]
so that $I$ is the first index starting from which all entries are
equal. 
\begin{enumerate}
\item \emph{\uline{$I=1$}}. In this case $\mvec$ is a constant sequence,
i.e 
\[
\mvec=\left(m_{1},\ldots,m_{1}\right).
\]
Choose 
\[
\mvec'=\left(m_{1}+1,0,\ldots,0\right).
\]
We have 
\[
\lambda_{\mvec}=m_{1}^{2}\sum_{j=1}^{n}\ohm^{2\left(j-1\right)}=m_{1}^{2}\frac{3}{\ohm^{2}-1}\quad\textrm{and }\quad\lambda_{\mvec'}=\left(m_{1}+1\right)^{2}.
\]
An easy calculation shows that $\lambda_{\mvec'}<\lambda_{\mvec}$
holds for all values of $n$ and $m_{1}$ with the only exceptions
being $\mvec=\left(1,1\right)$ and $\vec{m}=\vec{0}$. Indeed, we
see later (Lemma \ref{lem:CS_of_box}) that those are Courant-sharp.\\

\item \uline{$I\geq3$}. With 
\[
\mvec=(m_{1},\ldots,\underbrace{m_{I},}_{I-th~position}\ldots,m_{I}),
\]
choose 
\[
\mvec'=(0,..,0,\underbrace{m_{I}+1,}_{I-th~position}0,..,0)
\]
and consider an auxiliary point 
\[
\mvec'':=(m_{I}+1,\ldots,m_{I}+1,\underbrace{0,}_{I}..,0).
\]
Clearly we have 
\[
\lambda_{\mvec''}\leq\lambda_{\mvec},
\]
and it is left to show
\begin{equation}
\lambda_{\mvec'}<\lambda_{\mvec''}.\label{eq:Auxilary point for k geq 3}
\end{equation}
To get \eqref{eq:Auxilary point for k geq 3} simply note that 
\[
\ohm^{2\left(k-1\right)}\left(m_{I}+1\right)^{2}<\frac{1-\ohm^{2\left(k-1\right)}}{1-\ohm^{2}}\left(m_{I}+1\right)^{2}=\sum_{j=1}^{k-1}\ohm^{2\left(j-1\right)}\left(m_{I}+1\right)^{2},
\]
 for all $n\geq k\geq3$.\\

\item \uline{$I=2$}. With 
\[
\mvec=(m_{1},m_{2},\ldots,m_{2}),
\]
choose 
\[
\mvec'=(0,m_{2}+1,0,..,0)
\]
and consider an auxiliary point 
\[
\mvec'':=(m_{2}+1,m_{2},...,m_{2}).
\]
As
\[
\lambda_{\mvec''}\leq\lambda_{\mvec},
\]
it is left to show
\begin{equation}
\lambda_{\mvec'}<\lambda_{\mvec''}.\label{eq:Auxilary point for k equal 2}
\end{equation}
We get that \eqref{eq:Auxilary point for k equal 2} is equivalent
to 
\begin{equation}
m_{2}>\sqrt{\ohm^{2}-1}\left(\sqrt{\sum_{j=2}^{n}\ohm^{2\left(j-1\right)}}-\sqrt{\ohm^{2}-1}\right)^{-1}.\label{eq:condition_on_m2}
\end{equation}
This inequality holds if either $n\geq3$ and $m_{2}\geq1$ or $n=2$
and $m_{2}\geq3$ (the case $n=2,~m_{2}=3$ is demonstrated in Figure
\ref{fig:B_Q_and_Q_lambda}). \\
The remaining subcases are as follows

\begin{enumerate}
\item For $n\geq3,~m_{2}=0$ we have that $\mvec=\left(m_{1},0,\ldots,0\right)$
and 
\[
\lambda_{\mvec'}=\lambda_{(0,1,0,\ldots,0)}<\lambda_{(m_{1},0,\ldots,0)}=\lambda_{\mvec},
\]
for all $m_{1}\geq2$. \\
Note that $n\geq3,~m_{2}=0,~m_{1}\in\left\{ 0,1\right\} $ correspond
to Courant-sharp eigenvalues (Lemma \ref{lem:CS_of_box}).
\item For $n=2,~m_{2}\in\left\{ 0,1,2\right\} $, we have the following
subcases

\begin{enumerate}
\item If $m_{2}=2$ and $m_{1}>3$ then $\lambda_{\mvec'}=\lambda_{\left(0,3\right)}=18<m_{1}^{2}+8=\lambda_{\left(m_{1},2\right)}.$
\item If $m_{2}=2$ and $m_{1}=3$ then $\lambda_{\mvec'}=\lambda_{\left(4,0\right)}=16<17=\lambda_{\left(3,2\right)}.$
\item If $m_{2}=1$ and $m_{1}\geq3$ then $\lambda_{\mvec'}=\lambda_{\left(0,2\right)}=8<m_{1}^{2}+2=\lambda_{\left(m_{1},1\right)}.$
\item If $m_{2}=0$ and $m_{1}\geq2$ then $\lambda_{\mvec'}=\lambda_{\left(0,1\right)}=2<m_{1}^{2}=\lambda_{\left(m_{1},0\right)}.$\\
~
\end{enumerate}

Note that $\lambda_{\left(1,0\right)}$ and $\lambda_{\left(2,1\right)}$
correspond to Courant-sharp eigenvalues (Lemma \ref{lem:CS_of_box}).

\end{enumerate}
\end{enumerate}
\end{proof}
Finally, Theorem \ref{thm:CourantSharp_boxes} is proven by validating
the Courant-sharpness of the remaining eigenvalues.
\begin{lem}
\label{lem:CS_of_box}~
\begin{enumerate}
\item Let $n\geq3$. $\lambda_{1},~\lambda_{2}$ are Courant-sharp eigenvalues
of $\nbox$.
\item For n=2 (the rectangle case), $\lambda_{1},\lambda_{2},\lambda_{4},\lambda_{6}$
are Courant-sharp eigenvalues of $\mathcal{B}^{\left(2\right)}$.
\end{enumerate}
\end{lem}
\begin{proof}
By Courant's bound and orthogonality of eigenfunctions $\lambda_{1},\lambda_{2}$
are always Courant-sharp. For the rectangle, $\mathcal{B}^{\left(2\right)}$,
one counts that the eigenfunction $\varphi_{\left(1,1\right)}$, which
corresponds to $\lambda_{4}$ has four nodal domains and the eigenfunction
$\varphi_{\left(2,1\right)}$ which corresponds to $\lambda_{6}$
has six nodal domains.
\end{proof}
\appendix

\section{On multiplicity of eigenvalues of high-dimensional boxes\label{sec:appendix_multiplicity_boxes}}

We start by relating the multiplicity function to the following classical
problem. Denote the sum of squares function by 
\[
r_{2}(z)=\left|\{(m,n)\in\Z^{2}\text{ such that }m^{2}+n^{2}=z\}\right|.
\]
The equality $\mult\left(\lambda\right)=\mult\left(\oval\right)$
in Corollary \ref{cor:eiganvalue_is_odd_times_power_of_two} implies
that $r_{2}(z)=r_{2}(2^{k}z),$ for all $k\in\N$, (see also \cite{sum_of_squares_grosswald},
Chapter 2, Section 4). This fact nicely generalizes in Corollary \ref{cor:eiganvalue_is_odd_times_power_of_gamma}
by the same equality. Indeed, defining the following quadratic form,
\[
q(x_{1},..x_{n})=\sum_{j=1}^{n}\left(\ohm^{j-1}x_{j}\right)^{2},
\]
and denoting 
\[
r_{n}^{q}(z)=\left|\{(m_{1},..,m_{n})\in\Z^{n}\text{ such that }q(m_{1},..,m_{n})=z\}\right|,
\]
we get $r_{n}^{q}(z)=r_{n}^{q}(\ohm^{2k}z)$ for all $k\in\N$. In
particular, this relation seems more interesting for even values of
$n$, as can be interpreted from the following.
\begin{prop}
\label{prop:eigenvalue_multiplicity_box} ~
\begin{enumerate}
\item If $n$ is odd then all eigenvalues $\lambda\in\spec{\nbox}$ are
simple.
\item Let $n$ be even, and let $\lambda\in\spec{\nbox}$ . Then

\begin{enumerate}
\item \label{enu:prop_eigenvalue_multiplicity_box_2a}$\lambda$ is uniquely
written as $\lambda=\sum_{j=0}^{\frac{n}{2}-1}\lambda^{\left(j\right)}\ohm^{2j}$.
\item \label{enu:prop_eigenvalue_multiplicity_box_2b}Each $\lambda^{\left(j\right)}$
is some eigenvalue of the rectangle problem \emph{(}$\lambda^{\left(j\right)}\in\spec{\mathcal{B}^{\left(2\right)}}$\emph{)}.
\item \label{enu:prop_eigenvalue_multiplicity_box_2c}The multiplicity of
$\lambda$ equals to the product over multiplicities of all $\lambda^{\left(j\right)}$'s
as eigenvalues of the rectangle problem.
\end{enumerate}
\end{enumerate}
\end{prop}
\begin{proof}
~
\begin{enumerate}
\item Assume $n$ is odd. Assume that there exist $\mvec^{\left(1\right)},\mvec^{\left(2\right)}\in\Q$
such that $\lambda_{\mvec^{\left(1\right)}}=\lambda_{\mvec^{\left(2\right)}}$.
We get
\begin{align*}
\sum_{j=1}^{n}\left(\ohm^{j-1}m_{j}^{\left(1\right)}\right)^{2} & =\sum_{j=1}^{n}\left(\ohm^{j-1}m_{j}^{\left(2\right)}\right)^{2}~~~~\Leftrightarrow\\
0=\sum_{j=1}^{n}\ohm^{2\left(j-1\right)}\left(\left(m_{j}^{\left(1\right)}\right)^{2}-\left(m_{j}^{\left(2\right)}\right)^{2}\right) & =\\
=\sum_{j=1}^{\frac{n+1}{2}}\ohm^{2\left(j-1\right)}\left(\left(m_{j}^{\left(1\right)}\right)^{2}-\left(m_{j}^{\left(2\right)}\right)^{2}\right) & +\sum_{j=1}^{\frac{n-1}{2}}2\ohm^{2j-1}\left(\left(m_{\frac{n+1}{2}+j}^{\left(1\right)}\right)^{2}-\left(m_{\frac{n+1}{2}+j}^{\left(2\right)}\right)^{2}\right),
\end{align*}
using $\ohm^{n}=2$, in the reordering of terms in the last line.
The right-hand side of the above is a linear combination of the basis
$\left\{ \ohm^{j}\right\} _{j=0}^{n-1}$, so that we conclude for
all $j$, $~m_{j}^{\left(1\right)}=m_{j}^{\left(2\right)}$, as required.
\item Assume $n$ is even. Let $\mvec\in\Q$ such that $\lambda_{\mvec}=\lambda$.
We have,
\begin{align}
\lambda=\sum_{j=1}^{n}\left(\ohm^{j-1}m_{j}\right)^{2} & =\sum_{j=1}^{\frac{n}{2}}\ohm^{2\left(j-1\right)}\left(m_{j}^{2}+2m_{j+\frac{n}{2}}^{2}\right),\label{eq:evalue_box_even_dimension}
\end{align}
 and conclude
\begin{equation}
\forall1\leq j\leq\nicefrac{n}{2}\;\;\;\;\;\lambda^{\left(j-1\right)}=m_{j}^{2}+2m_{j+\frac{n}{2}}^{2}.\label{eq:evalue_box_even_dimension_coefficient}
\end{equation}
 Hence we have shown \eqref{enu:prop_eigenvalue_multiplicity_box_2a},
where the uniqueness comes from $\left\{ \ohm^{2j}\right\} _{j=0}^{\nicefrac{n}{2}-1}$
being a basis. From \eqref{eq:evalue_box_even_dimension_coefficient}
it is easily verified that $\lambda^{\left(j-1\right)}\in\spec{\mathcal{B}^{\left(2\right)}}$
and combining with \eqref{eq:evalue_box_even_dimension} we deduce
that the multiplicity of $\lambda\in\spec{\nbox}$ is obtained as
a product over all multiplicities of $\left\{ \lambda^{\left(j-1\right)}\right\} _{j=1}^{\nicefrac{n}{2}}$.
\end{enumerate}
\end{proof}
\begin{rem*}
The second part of the proposition above may be explained as following.
One may express $\mathcal{B}^{(2k)}$ as a direct product of $k$
scaled copies of $\mathcal{B}^{(2)}$. Denoting the edge lengths of
$\mathcal{B}^{(2k)}$ by $l_{1},\ldots l_{2k}$, the $j^{\textrm{th }}$
copy of $\mathcal{B}^{(2k)}$ has edge lengths $l_{j},l_{k+j}$. Each
eigenfunction on $\mathcal{B}^{(2k)}$ can be expressed as a product
of eigenfunctions on all different $k$ scaled copies of $\mathcal{B}^{(2)}$.
Hence each eigenvalue of $\mathcal{B}^{(2k)}$ is a sum over eigenvalues
of all the $\mathcal{B}^{(2)}$ copies.
\end{rem*}

\section{nodal deficiency\label{sec:appendix_nodal_deficiency}}

The importance of nodal deficiency of eigenfunctions has been recognized
in recent studies \cite{BanBerRazSmi_cmp12,Ber_apde13,BerKucSmi_gafa12,BerRazSmi_jpa12,BerWey_ptrsa13}.
It is the nodal deficiency that has been exactly expressed by variations
over partitions and eigenvalues. These recent works concern manifolds,
as well as quantum and discrete graphs. We bring here some interesting
bounds on the nodal deficiencies of the spectral problems studied
in this paper.

We define the nodal deficiency of an eigenfunction $\varphi$ of an
eigenvalue $\lambda$ by 
\[
\nodaldef{\varphi}=N\left(\lambda\right)-\nu\left(\varphi\right),
\]
and the nodal deficiency of an eigenvalue by 
\begin{equation}
\nodaldef{\lambda}:=\underset{\varphi\in E\left(\mbox{\ensuremath{\lambda}}\right)}{\min}\nodaldef{\varphi}.\label{eq:nodal deficiect}
\end{equation}
Where $E\left(\lambda\right)$ is the eigenspace associated to $\lambda.$
In the following we use the analysis of Section \ref{sec:proof_of_triangle}
to derive lower bounds of the nodal deficiency of eigenvalues. The
following lemmata hold for all the domains treated in the paper. We
use the notation $\Omega$ to indicate both $\trngle$ and $\nbox$
and denote
\[
\ohmdmn\left(\dom\right):=\begin{cases}
\sqrt{2} & \dom=\trngle\\
\ohm & \dom=\B^{(n)}
\end{cases}.
\]
The next lemma provides a lower bound on the nodal deficiency of multiple
eigenvalues.
\begin{lem}
\label{lem:appendix_nodal_deficiency}Let $\oval\in\ospec{\dom}$
and $k\in\N_{0}$. The nodal deficiency of the eigenvalue $\ohmdmn\left(\dom\right)^{2k}\cdot\oval$
obeys 
\[
\nodaldef{\ohmdmn\left(\dom\right)^{2k}\cdot\oval}\geq\left(\mult\left(\oval\right)-1\right)\cdot\left(M\left(k,\dom\right)-1\right),
\]

\noindent where $M(k,\dom)$ is the number of the subdomains of the
$k$-frame partition of $\dom$ and $d\left(\oval\right)$ is the
multiplicity of $\oval$.\end{lem}
\begin{proof}
For the sake of convenience, we abbreviate notations by writing $\ohmdmn$
instead of $\ohmdmn\left(\dom\right)$ and $M\left(k\right)$ instead
of $M(k,\dom)$. Note that the following arguments below are similar
to those we have used in the proof of Lemma \ref{lem:CS_implies_simplicity}.
We have
\[
\nu\left(\varphi\right)=\sum_{i=1}^{M(k)}\nu\left(\varphi\at_{\dom_{i}^{(k)}}\right)\underbrace{\leq}_{\textrm{Courant nodal theorem}}\sum_{i=1}^{M\left(k\right)}N_{i}^{\left(k\right)}\left(\ohmdmn^{2k}\cdot\oval\right)
\]
\[
=\sum_{i=1}^{M\left(k\right)}\nup_{i}^{\left(k\right)}\left(\ohmdmn^{2k}\cdot\oval\right)+\sum_{i=1}^{M\left(k\right)}\left(N_{i}^{\left(k\right)}\left(\ohmdmn^{2k}\cdot\oval\right)-\nup_{i}^{\left(k\right)}\left(\ohmdmn^{2k}\cdot\oval\right)\right)
\]
\[
\underbrace{\leq}_{\textrm{variational principle \eqref{eq:variational principle} }}\nup\left(\ohmdmn^{2k}\cdot\oval\right)+\sum_{i=1}^{M\left(k\right)}\left(N_{i}^{\left(k\right)}\left(\ohmdmn^{2k}\cdot\oval\right)-\nup_{i}^{\left(k\right)}\left(\ohmdmn^{2k}\cdot\oval\right)\right)
\]
\[
\underbrace{\leq}_{\textrm{Lemma \ref{lem:degeneracy of subdomain}}}\nup\left(\ohmdmn^{2k}\cdot\oval\right)+\sum_{i=1}^{M\left(k\right)}\left(N\left(\ohmdmn^{2k}\cdot\oval\right)-\nup\left(\ohmdmn^{2k}\cdot\oval\right)\right)
\]
\[
=N\left(\ohmdmn^{2k}\cdot\oval\right)+\left(M\left(k\right)-1\right)\cdot\left(N\left(\ohmdmn^{2k}\cdot\oval\right)-\nup\left(\ohmdmn^{2k}\cdot\oval\right)\right).
\]
Thus

\[
N\left(\ohmdmn^{2k}\cdot\oval\right)-\nu\left(\varphi\right)\geq\left(M\left(k\right)-1\right)\cdot\left(\nup\left(\ohmdmn^{2k}\cdot\oval\right)-N\left(\ohmdmn^{2k}\cdot\oval\right)\right),
\]
and 
\[
\nodaldef{\varphi}\geq\left(\mult\left(\ohmdmn^{2k}\cdot\oval\right)-1\right)\cdot\left(M\left(k\right)-1\right).
\]
Since the right-hand side does not depend on $\varphi$ we get 
\[
\nodaldef{\ohmdmn^{2k}\cdot\oval}\geq\left(\mult\left(\ohmdmn^{2k}\cdot\oval\right)-1\right)\cdot\left(M\left(k\right)-1\right).
\]
Finally, use 
\[
\mult\left(\oval\right)=\mult\left(\ohmdmn^{2k}\cdot\oval\right),
\]
 (see Corollary \ref{cor:eiganvalue_is_odd_times_power_of_two} for
the triangle or Corollary \ref{cor:eiganvalue_is_odd_times_power_of_gamma}
for the boxes) to finish the proof.
\end{proof}
We may obtain even more explicit bounds on the nodal deficiency by
computing $M\left(k,\dom\right)$, as explained in the following.
In the case of $\B^{(n)}$, we can get an explicit expression for
$M\left(k,\B^{(n)}\right)$, noticing the relations 
\begin{align*}
\forall0\leq k\leq n-1;~~~M\left(k,\B^{(n)}\right) & =2\\
\forall k\geq0;~~~M\left(k+n,\B^{(n)}\right)= & 2M\left(k,\B^{(n)}\right)-1.
\end{align*}
Those relations may be obtained by noticing that all $k$-frames are
formed by hyperplanes, all parallel to each other. The number of those
hyperplanes determines $M\left(k,\B^{(n)}\right)$ and this number
may be deduced by working out the definition of $k$-frames (Definition
\ref{def:k_frame_partition} with \eqref{eq:hyperplabe_of_box},\eqref{eq:reflection_box},\eqref{eq:unfold_coords_box},
and see as an example Figure \ref{fig:four_unfoldings_rectangle}).
From those relations we obtain 
\[
M\left(k,\B^{(n)}\right)=2^{\left\lfloor \frac{k}{n}\right\rfloor }+1,\,\,\forall k\geq0.
\]
 In the case of the triangle, we may also obtain the explicit expression
for $M\left(k,\trngle\right)$. Yet, as the calculation is somewhat
cumbersome, we chose to provide the following estimate. A square subdomain
appears on the 4-frame partition (see Figure \ref{fig:four_unfoldings-1}).
Getting to the next $k$-frames, each unfolding at least doubles the
number of this particular subdomain (up to scaling) and hence $M\left(k,\trngle\right)>c2^{k}$,
for some constant $c$. In effect, the exact calculation gives the
same order of magnitude, i.e. $M\left(k,\trngle\right)=\Theta(2^{k}).$
Applying Lemma \ref{lem:appendix_nodal_deficiency} for odd eigenvalues,
where $k=0$ and $M\left(0,\Omega\right)=2$, we get $\nodaldef{\oval}\geq\mult\left(\oval\right)-1$.
We may actually improve this bound by relating the nodal deficiency
with the count of boundary lattice points, as follows.
\begin{lem}
Let $\oval\in\ospec{\dom}$ , then
\[
\nodaldef{\oval}\geq\left|\bdry\left(\oval\right)\cap\E\right|-1.
\]
\end{lem}
\begin{proof}
Let $\varphi$ be an eigenfunction that corresponds to $\oval\in\ospec{\dom}.$
For the triangle we have by Equation \eqref{eq:nodal deficiecy of odd eigenvalues}
that 
\[
\nodaldef{\varphi}\geq\left|\bdry\left(\oval\right)\cap\E\right|-1,
\]
and the same bound for the boxes, defining $\bdry\left(\oval\right)$
by generalizing \eqref{eq:right_boundary_points}. The lemma now follows
since the right-hand side does not depend on $\varphi$.
\end{proof}
Note that in the course of the proof of Proposition \ref{prop:rulling_all_courant_sharp_traingle},\eqref{enu:prop_zeroth_unfoldings_not_CS}
(see \eqref{eq:boundary_has_at_least_two_points}) it is shown that
$\left|\bdry\left(\oval\right)\cap\E\right|>1$, so that the bound
of the lemma above is not trivial. In fact, by a lattice analysis
one may further get that the size of this set is of order $\sqrt{\oval}$.

\section{The Dirichlet problem\label{sec:appendix_dirichlet_problem}}

We shortly discuss below how the methods of this work may be applied
to examine the Dirichlet eigenvalue problem on the domains treated
herein. The Courant-sharp eigenvalues of the Dirichlet right-angled
isosceles are already determined in \cite{BerHel_lmp16} using the
analysis of the corresponding Dirichlet eigenvalue problem, \cite{AroBanFajGnu_jpa12},
done by two of the authors of the current paper together with Aronovitch
and Gnutzmann. 

Let us lay the framework for examining the Dirichlet 2-rep-tiles.
The quantum number set of the Dirichlet triangle is 
\begin{equation}
\Q:=\left\{ \left.\left(m,n\right)\in\mathbb{N}\times\mathbb{N}\right|m>n\right\} ,\label{eq:quantum_numbers triangle dirichlet}
\end{equation}
and the corresponding eigenvalues are 
\[
\lambda_{m,n}=\left\Vert \left(m,n\right)\right\Vert ^{2}~~;\ \left(m,n\right)\in\Q.
\]
For the boxes the quantum number set is 
\begin{equation}
\Q:=\left\{ \mvec\in\mathbb{N}^{n}\right\} ,\label{eq:quantum_numbers_box-1}
\end{equation}
and the corresponding eigenvalues are
\begin{equation}
\lambda_{\mvec}=\sum_{j=1}^{n}\left(\ohm^{j-1}m_{j}\right)^{2}~~;\ \mvec\in\Q.\label{eq:eigenvalue_box-1}
\end{equation}
Note that these sets of quantum numbers are included in those defined
for the Neumann problems. We exploit this to define $\O$, $\E$ and
in turn $\ospec{\Omega}$ and $\espec{\Omega}$ for $\Omega$ being
either $\trngle$ or $\B^{(n)}$ exactly in the same manner as we
did for the Neumann problem (see Definitions \ref{def:odd_even_spectrum}
and \ref{def:odd_even_spectrum_box}). We obtain for the Dirichlet
problem the following lemma, which may be prove similarly to its Neumann
analogues, Lemmata \ref{lem:eigenfunction_symmetry} and \ref{lem:eigenfunction_symmetry_box}.
\begin{lem}
\label{lem:eigenfunction_symmetry_dirichlet}Let $\lambda\in\ospec{\Omega}$
\emph{(}$\lambda\in\espec{\Omega}$\emph{)}, then its corresponding
eigenfunctions are even \emph{(}odd\emph{)} w.r.t. $\line$\emph{
}if and only if $\lambda$ is odd \emph{(}even\emph{).}
\end{lem}
One should pay careful attention to the difference in phrasing of
this lemma comparing to its Neumann analogues. Here an eigenvalue
belongs to the \emph{even} spectrum if and only if its eigenfunctions
are \emph{odd}. In turn, the folding transformation may be applied
only on even eigenvalues (alternatively, on odd eigenfunctions).

We use Lemma \ref{lem:eigenfunction_symmetry_dirichlet} to express
the nodal deficiency in terms of boundary lattice points, adopting
the notation $\ohmdmn\left(\Omega\right)$ of the previous appendix.
\begin{lem}
Let $\lambda\in\espec{\Omega}$ then we have
\[
\nodaldef{\lambda}=2\cdot\nodaldef{\ohmdmn\left(\dom\right)^{-2}\cdot\lambda}+\left|\bdry\left(\lambda\right)\cap\O\right|-1.
\]
\end{lem}
\begin{proof}
Start by noting that $\left.U_{\Q}\right|_{\Q\left(\ohmdmn\left(\dom\right)^{-2}\cdot\lambda\right)}$
is an injection and its image is $\E\left(\lambda\right),$ thus 
\begin{equation}
\left|\E\left(\lambda\right)\right|=\ndown\left(\ohmdmn\left(\dom\right)^{-2}\cdot\lambda\right).\label{eq:even and lower counting}
\end{equation}
Also note that 
\[
B:\,\,\mathcal{\E}\left(\lambda\right)\rightarrow\mathcal{O}\left(\lambda\right)\backslash\left(\underrightarrow{\partial}\Q\left(\lambda\right)\cap\mathcal{O}\right)
\]
\[
B(m_{1},...,m_{n})=(m_{1}-1,...,m_{n}),
\]
with $n=2$ in the triangle case, is a bijection (cf. \eqref{eq:boundary_map}),
which gives
\begin{equation}
\left|\O\left(\lambda\right)\right|=\left|\E\left(\lambda\right)\right|+\left|\bdry\left(\lambda\right)\cap\O\right|.\label{eq:odd is bigger than even}
\end{equation}
We get 
\[
N\left(\lambda\right)=\left|\E\left(\lambda\right)\right|+\left|\O\left(\lambda\right)\right|+1\underbrace{=}_{\eqref{eq:even and lower counting},\,\eqref{eq:odd is bigger than even}}2\cdot N\left(\ohmdmn\left(\dom\right)^{-2}\cdot\lambda\right)+\left|\bdry\left(\lambda\right)\cap\O\right|-1.
\]

Choosing some eigenfunction $\varphi$ of $\lambda$ we have

\begin{equation}
N\left(\lambda\right)=\nu\left(\varphi\right)+\nodaldef{\varphi}\underbrace{=}_{\varphi\,\text{is odd}}2\cdot\nu\left(\varphi\at_{\frac{1}{2}\Omega}\right)+\nodaldef{\varphi}.\label{eq:counting is nodals plus deficiency}
\end{equation}
Hence
\[
2\cdot N\left(\ohmdmn\left(\dom\right)^{-2}\cdot\lambda\right)+\left|\bdry\left(\lambda\right)\cap\O\right|-1=2\nu\left(\varphi\at_{\frac{1}{2}\Omega}\right)+\nodaldef{\varphi},
\]
which in turn leads to
\begin{equation}
\nodaldef{\varphi}=2\cdot\nodaldef{\varphi\at_{\frac{1}{2}\Omega}}+\left|\bdry\left(\lambda\right)\cap\O\right|-1.\label{eq:dirichlet eigenfunction deficiency identity}
\end{equation}
Since the nodal deficiency of an eigenvalue is the minimal deficiency
over all corresponding eigenfunctions we get 
\[
\nodaldef{\varphi}\geq2\cdot\nodaldef{\ohmdmn\left(\dom\right)^{-2}\cdot\lambda}+\left|\bdry\left(\lambda\right)\cap\O\right|-1.
\]
The right-hand side is independent of $\varphi$ and therefore
\[
\nodaldef{\lambda}\geq2\cdot\nodaldef{\ohmdmn\left(\dom\right)^{-2}\cdot\lambda}+\left|\bdry\left(\lambda\right)\cap\O\right|-1.
\]
We note that the opposite inequality follows by the same method, which
finishes the proof.
\end{proof}
We end by noting that as $\left|\bdry\left(\oval\right)\cap\E\right|>1$
(see \eqref{eq:boundary_has_at_least_two_points} and discussion at
the end of the previous appendix) the result of the lemma both supplies
a non-trivial bound on the deficiency and also rules out all even
eigenvalues from being Courant-sharp (the argument is actually the
same as the one used in the proof of Proposition \ref{prop:rulling_all_courant_sharp_traingle},\eqref{enu:prop_zeroth_unfoldings_not_CS}
to rule out the Courant-sharpness of odd eigenvalues in the Neumann
case).

\section*{Acknowledgments}

We wish to thank Thomas Hoffmann-Ostenhof for helpful discussions
and Danny Neftin for his algebraic remarks. Sebastian Egger is warmly
acknowledged for the the careful reading of the manuscript. R.B. and
M.B. were supported by ISF (Grant No. 494/14). R.B. was supported
by Marie Curie Actions (Grant No. PCIG13-GA-2013-618468) and the Taub
Foundation (Taub Fellow). D.F. thanks the mathematics faculty of the
Technion for their hospitality. 

\bibliographystyle{plain}
\bibliography{CourantSharpness}

\end{document}